\documentclass[11pt,a4paper,british]{amsart}
\usepackage[T1]{fontenc}
\usepackage[utf8]{inputenc}
\usepackage{varioref}
\usepackage{prettyref}
\usepackage{setspace}
\usepackage{amssymb}

\makeatletter
\numberwithin{equation}{section} 
\numberwithin{figure}{section} 
  \@ifundefined{theoremstyle}{\usepackage{amsthm}}{}
  \theoremstyle{plain}
  \newtheorem{thm}{Theorem}[section]
  \theoremstyle{plain}
  \newtheorem{conjecture}[thm]{Conjecture}
  \theoremstyle{definition}
  \newtheorem{defn}[thm]{Definition}
  \theoremstyle{remark}
  \newtheorem{rem}[thm]{Remark}
  \theoremstyle{plain}
  \newtheorem{prop}[thm]{Proposition}
  \theoremstyle{plain}
  \newtheorem{lem}[thm]{Lemma}
  \theoremstyle{plain}
  \newtheorem{cor}[thm]{Corollary}
  \theoremstyle{remark}
  \newtheorem{note}[thm]{Note}

\usepackage{cancel}

\usepackage[varg]{pxfonts}


\makeatother

\usepackage{babel}

\begin{document}

\title[Rational Dilation on Symmetric Domains]{Counterexamples to Rational Dilation on Symmetric Multiply Connected
Domains}

\author{James Pickering}

\address{Department of Mathematics\\
 University of Newcastle-upon-Tyne\\
 Newcastle-upon-Tyne\\
 United Kingdom\\
 NE1 7RU}

\email{james.pickering@ncl.ac.uk}

\thanks{This paper is based on work contributing towards the author's PhD
thesis at the University of Newcastle-upon-Tyne, under the supervision
of Michael Dritschel. This work is funded by the Engineering and Physical
Sciences Research Council.}


\date{Submitted on \today}

\begin{abstract}
We show that if $R$ is a compact domain in the complex plane with
two or more holes and an anticonformal involution onto itself (or
equivalently a hyperelliptic Schottky double), then there is an operator
$T$ which has $R$ as a spectral set, but does not dilate to a normal
operator with spectrum on the boundary of $R$. 
\end{abstract}
\maketitle

\subsection{Definitions}

Let $X$ be a compact, path connected subset of $\mathbb{C}$, with
interior $R$, and \emph{analytic boundary} $B$ composed of $n+1$
disjoint curves, $B_{0},\,\ldots,\, B_{n}$, where $n\geq2$. By analytic
boundary, we mean that for each boundary curve $B_{i}$ there is some
biholomorphic map $\phi_{i}$ on a neighbourhood $U_{i}$ of $X$
which maps $B_{i}$ to the unit circle $\mathbb{T}$. By convention
$B_{0}$ is the outer boundary. We write $\Pi=B_{0}\times\cdots\times B_{n}$.

We say a Riemann surface $Y$ is \emph{hyperelliptic} if there is
a meromorphic function with two poles on $Y$ (see \cite{FarkasKra}).
We say $R$ is symmetric if there exists some anticonformal involution
$\varpi$ on $R$ with $2n+2$ fixed points on $B$. We say a domain
in $\mathbb{C}\cup\{\infty\}$ (that is, the Riemann sphere $S^{2}$)
is a \emph{real slit domain} if its complement is a finite union of
closed intervals in $\mathbb{R}\cup\{\infty\}$.

We define $\mathcal{R}(X)\subseteq C(X)$ as the space of all rational
functions that are continuous on $X$. The definitions of contractivity
and complete contractivity are the usual definitions, and can be found
in \cite{Paulsen}.

\subsection{Introduction}

A key problem that this paper deals with is the \emph{rational dilation
conjecture}, which is as follows.

\begin{conjecture}
If $X\subseteq\mathbb{C}$ is a compact domain, $T\in\mathcal{B}(H)$
is a Hilbert space operator with $\sigma(T)\subseteq X$ and $\left\Vert f(T)\right\Vert \leq1$
for all $f\in\mathcal{R}(X)$, then there is some normal operator
$N\in\mathcal{B}(K)$, $K\supseteq H$, such that $\sigma(N)\subseteq B\:(=\partial X)$,
and $f(T)=P_{H}N\vert_{H}$. 
\end{conjecture}
A classical result of Sz.-Nagy shows that the rational dilation conjecture
holds if $X$ is the unit disc. A generalisation by Berger, Foias
and Lebow shows this holds for any simply connected domain (see \cite{Paulsen}).
A result by Agler (see \cite{AglerAnnulus}) shows that rational dilation
also holds if $X$ has one hole -- such as in an annulus. However,
subsequent work has shown that rational dilation fails on every two-holed
domain with analytic boundary (see \cite{DritschelRationalDilation},
and \cite{AglerComputation}).

The aim of this paper is to prove the following, which by a result
of Arveson (see \cite[Cor. 7.8]{Paulsen}), is equivalent to showing
that the rational dilation conjecture does not hold on any symmetric,
two-or-more-holed domain.

\begin{thm}
\label{thm:ElTheoremGrande}If $X$ is a symmetric domain in $\mathbb{C}$,
with $2\leq n<\infty$ holes, there is an operator $T\in\mathcal{B}(H)$,
for some Hilbert space $H$, such that the homomorphism $\pi:\mathcal{R}(X)\to\mathcal{B}(H)$
with $\pi\left(p/q\right)=p(T)\cdot q(T)^{-1}$ is contractive, but
not completely contractive. 
\end{thm}
\begin{proof}[Proof Outline]First, we let $\mathcal{C}$ define the cone generated
by \[
\left\{ H(z)\left[1-\psi(z)\overline{\psi(w)}\right]H(w)^{*}:\,\psi\in\mathcal{B}\mathbb{H}(x),\, H\in M_{2}\left(\mathbb{H}(X)\right)\right\} \,,\]
 where $\mathcal{B}\mathbb{H}(X)$ is the unit ball of the space of
functions analytic in a neighbourhood of $X$, under the supremum
norm, and $M_{2}\left(\mathbb{H}(X)\right)$ is the space of $2\times2$
matrix valued functions analytic in a neighbourhood of $X$. For $F\in M_{2}\left(\mathbb{H}(X)\right)$,
we set \[
\rho_{F}=\sup\left\{ \rho>0:\, I-\rho^{2}F(z)F(w)^{*}\in\mathcal{C}\right\} \,.\]
 We show that there exists a function $F$ which is unitary valued
on $B$ (we say $F$ is \emph{inner}), but such that $\rho_{F}<1$.
We show that such a function generates a counter-example of the type
needed. To show that such a function exists, we show that if $F$
is inner, $\rho_{F}=1$ ($\left\Vert F\right\Vert =1$ by the max
modulus principle, so $\rho_{F}\leq1$), and the zeroes of $F$ are
{}``well behaved'', then $F$ can be diagonalised. We go on to show
that there is a non-diagonalisable inner function $F$, with well
behaved zeroes, which must therefore have $\rho_{F}<1$, so must be
a counter-example. 
\end{proof}

\section{Symmetries}

Details of the ideas discussed below can be found in \cite{Barker}.
A less detailed (but more widely available) presentation can be found
in \cite{BarkerArticle}.

\begin{thm}
Let $R\subseteq\mathbb{C}$ have $n+1$ analytic boundary curves,
$B_{0},\,\ldots,\, B_{n}\subseteq B$, with $n\geq2$, and let $Y$
be its Schottky double\label{thm:RedGreen}. The following are equivalent:
\begin{enumerate}
\item $Y$ is hyperelliptic;
\item $R$ is symmetric;
\item $R$ is conformally equivalent \label{itm:SlitDomain}to a real slit
domain $\Xi$.
\end{enumerate}
\end{thm}
The proof can be found in \cite{Barker}, but we will briefly discuss
the constructions involved. We know from \cite[III.7.9]{FarkasKra}
that $Y$ is hyperelliptic if and only if there is a conformal involution
$\iota:Y\to Y$ with $2n+2$ fixed points. We find that $\iota$ is
given by \[
\iota(x)=\begin{cases}
J\circ\varpi(x) & x\in R\\
\varpi(x) & x\in B\\
\varpi\circ J(x) & x\in J(R)\end{cases}\,,\]
 where $J$ is the {}``mirror'' function on $Y$.

Also, if $\varsigma:\Xi\to R$ is the conformal mapping from part
\ref{itm:SlitDomain}, we have that $\varpi(\varsigma(\xi))=\varsigma\left(\overline{\xi}\right)$.

\begin{defn}
We define the \emph{fixed point set} of our symmetric domain $R$
as \[
\mathbb{X}:=\left\{ x\in R:\, x=\varpi(x)\right\} \,.\]

\end{defn}
\begin{rem}
In view of Theorem \vref{thm:RedGreen}, it makes sense to relabel
the components of $B$. We can see that $\mathbb{X}$ must be the
image of $\mathbb{R}\cap\Xi$ under $\varsigma$, so must consist
of a finite collection of paths running between fixed points of $B$.
We choose one of the two fixed points of $B_{0}$, and call it $p_{0}^{-}$.
We follow $\mathbb{X}$ from $p_{0}^{-}$ to another $B_{i}$ which
we relabel $B_{1}$; we call the fixed point we landed at $p_{1}^{+}$.
Label the other fixed point in $B_{1}$ as $p_{1}^{-}$, and repeat,
until we reach $p_{0}^{+}$. The section of $\mathbb{X}$ from $p_{i}^{-}$
to $p_{i+1}^{+}$, we call $\mathbb{X}_{i}$. 
\end{rem}
\begin{prop}
\label{pro:TopRedBox}If a meromorphic function on $Y$ has $n$ or
fewer poles, and all of these poles lie in $R\cup B$, then all of
these poles must lie on $B$. 
\end{prop}
\begin{proof}
Suppose $f$ has $n$ or fewer poles. Then $f\circ\iota$ also has
$n$ or fewer poles, so $f-f\circ\iota$ has $2n$ or fewer poles.
However, if $x$ is a fixed point of $\iota$, $f(x)-f\circ\iota(x)=0$,
and since $\iota$ has $2n+2$ fixed points, $f-f\circ\iota$ has
at least $2n+2$ zeroes. This is only possible if $f-f\circ\iota\equiv0$,
so if $x$ is a pole of $f$, then $\iota(x)$ is a pole of $f$,
which is a contradiction unless $x\in B$. 
\end{proof}

\section{Inner Functions}

Many of the ideas found in this section can also be found in \cite{AglerComputation}
and \cite{DritschelRationalDilation}.

Results in this section often require us to choose a point $b\in R$.
Usually, $b$ will be determined by the particular application, but
in this section we make no requirements on the choice of $b$.

\subsection{Harmonic and Analytic Functions}

If $\omega_{b}$ is harmonic measure at $b$, and $s$ is arc length
measure, by an argument like the one in \cite{DritschelRationalDilation},
we can find a Poisson kernel $\mathbb{P}:R\times B\to\mathbb{R}$
such that for $h$ harmonic on $R$ and continuous on $B$,\[
h(w)=\int_{B}h(z)\mathbb{P}(w,\, z)ds(z)\,.\]
 Equivalently, $\mathbb{P}$ is given by the Radon-Nikodym derivative
\[
\mathbb{P}(w,\,\cdot)=\frac{d\omega_{w}}{ds}\,.\]
 We know that $\mathbb{P}$ is harmonic in $R$ at each point in $B$,
and that for any positive $h$ harmonic on $R$, and continuous on
$X$ there exists some positive measure $\mu$ on $B$ such that \[
h(w)=\int_{B}\mathbb{P}(w,\, z)d\mu(z)\,.\]
 Conversely, given a positive measure $\mu$ on $B$, this formula
defines a positive harmonic function.

We let $h_{j}$ denote the solution to the Dirichlet problem which
is 1 on $B_{j}$ and 0 on $B_{i}$, where $i\neq j$. We can see that
this corresponds to the arc length measure on $B_{j}$.

We define $Q_{j}:B\to\mathbb{R}$ as the outward normal derivative
of $h_{j}$, and define the \emph{periods} of $h$ by \[
P_{j}(h)=\int_{B}Q_{j}d\mu\,.\]
 It should be clear that $h$ is the real part of an analytic function
if and only if $P_{j}(h)=0$ for $j=0,\,1,\,\ldots,\, n$.

\begin{lem}
The functions $Q_{j}$ have no zeroes on $B$. Moreover, $Q_{j}>0$
on $B_{j}$ and $Q_{j}<0$ on $B_{l}$ for $l\neq j$. 
\end{lem}
\begin{proof}
As $X$ has analytic boundary, we can assume without loss of generality
that $B_{0}=\mathbb{T}$. We know that $h_{j}$ takes its minimum
and maximum on its boundary. Since $h_{j}$ equals one on $B_{j}$,
and zero on $B_{l}$ if $l\neq j$, these must be its maximum and
minimum respectively, so $h_{j}$ is non-decreasing towards $B_{j}$,
and non-increasing towards $B_{l}$, so $Q_{j}\geq0$ on $B_{j}$
and $Q_{j}\leq0$ on $B_{l}$.

We can see by the above argument that we only need show that $Q_{j}\neq0$.
We let $R^{\prime}$ be the reflection of $R$ about $B_{0}$ (which
we are assuming is the unit circle). We can extend $h_{j}$ to a harmonic
function on $X\cup R^{\prime}$ by setting \[
h_{j}(z)=-h_{j}(1/\bar{z})\]
 on $R^{\prime}$.

If $Q_{j}$ had infinitely many zeroes on $B_{0}$, then $Q_{j}$
would be identically zero, so we suppose $Q_{j}$ has finitely many
zeroes on $B_{0}$.

Suppose $Q_{j}$ has a zero $z$, and a small, simply connected neighbourhood
$N(z)$. By choosing $N(z)$ small enough, we can ensure that $N(z)$
contains no other zeroes. Clearly, $h_{j}$ forms the real part of
some holomorphic function $f$ on $N(z)$. We know that $\partial h_{j}/\partial n=Q_{j}=0$,
and because $h_{j}$ is constant on $B_{0}$, we know that the tangential
derivative of $h_{j}$, $\partial h_{j}/\partial t$, is also zero,
so $f$ has derivative zero at $z$, so $f$ has a ramification of
order at least two at $z$. We also know that $f$ maps everything
outside the unit disc to the left half plane, and everything inside
the unit disc to the right half plane, but clearly this is impossible,
so $Q_{j}$ cannot have a zero.

A similar argument holds for $B_{1},\,\ldots,B_{n}$. 
\end{proof}
\begin{cor}
\label{cor:NotZeroOnBj}If $h$ is a non-zero positive harmonic function
on $R$ which is the real part of an analytic function, and $h$ is
represented in terms of a positive measure $\mu$, then $\mu(B_{j})>0$
for each $j$. 
\end{cor}
\begin{proof}
If $\mu(B_{j})=0$, then as $Q_{j}<0$ on $B\backslash B_{j}$, $P_{j}(h)<0$,
a contradiction. Thus, $\mu(B_{j})>0$. 
\end{proof}

\subsection{Some Matrix Algebra}

We wish to show that at each $p\in\Pi$, the vector \[
V^{n}=\det\left(\begin{array}{cccc}
\mathbf{e}_{0} & \mathbf{e}_{1} & \cdots & \mathbf{e}_{n}\\
Q_{1}(p_{0}) & Q_{1}(p_{1}) & \cdots & Q_{1}(p_{n})\\
Q_{2}(p_{0}) & Q_{2}(p_{1}) & \cdots & Q_{2}(p_{n})\\
\vdots & \vdots & \ddots & \vdots\\
Q_{n}(p_{0}) & Q_{n}(p_{1}) & \cdots & Q_{n}(p_{n})\end{array}\right)\]
 has only positive coordinates. It helps to note that in three dimensions
\[
\mathbf{x}\times\mathbf{y}=\det\left(\begin{array}{ccc}
\mathbf{e}_{0} & \mathbf{e}_{1} & \mathbf{e}_{2}\\
x_{0} & x_{1} & x_{2}\\
y_{0} & y_{1} & y_{2}\end{array}\right)\,.\]

It will also be helpful to write \[
V^{n}=\left|\begin{array}{cccccc}
\mathbf{e}_{0} & \mathbf{e}_{1} & \mathbf{e}_{2} & \mathbf{e}_{3} & \cdots & \mathbf{e}_{n}\\
- & + & - & - & \cdots & -\\
- & - & + & - & \cdots & -\\
- & - & - & + & \cdots & -\\
\vdots & \vdots & \vdots & \vdots & \ddots & \vdots\\
- & - & - & - & \cdots & +\end{array}\right|\,,\]
 noting that $Q_{j}(p_{j})>0$, and $Q_{i}(p_{j})<0$ for $i\neq j$.
From here on, positive and negative quantities will simply be denoted
by $(+)$ and $(-)$, respectively.

\begin{lem}
\label{lem:PosSquare}All sub-matrices of $V^{n}$ of the form \[
\left(\begin{array}{ccccc}
+ & - & - & \cdots & -\\
- & + & - & \cdots & -\\
- & - & + & \cdots & -\\
\vdots & \vdots & \vdots & \ddots & \vdots\\
- & - & - & \cdots & +\end{array}\right)\]
 have positive determinant. 
\end{lem}
\begin{proof}
We can assume, without loss of generality, that such matrices are
of the form \[
\left(\begin{array}{cccc}
Q_{1}(p_{1}) & Q_{1}(p_{2}) & \cdots & Q_{1}(p_{k})\\
Q_{2}(p_{1}) & Q_{2}(p_{2}) & \cdots & Q_{2}(p_{k})\\
\vdots & \vdots & \ddots & \vdots\\
Q_{k}(p_{1}) & Q_{k}(p_{2}) & \cdots & Q_{k}(p_{k})\end{array}\right):=A^{T}\]
 by a simple relabelling of boundary curves. We note that \[
\sum_{j=0}^{n}h_{j}\equiv1\,,\]
 so in particular \[
\sum_{j=0}^{n}Q_{j}(x)=0\]
 for all $x\in B$. So, if $1\leq i\leq k$, then \[
\sum_{j=1}^{k}Q_{j}(p_{i})=-\left(Q_{0}(p_{i})+\sum_{j=k+1}^{n}Q_{j}(p_{i})\right)>0\,.\]
 We now apply Gershgorin's circle theorem. Since $A_{ij}=Q_{j}(p_{i})$,
the eigenvalues of $A$ are in the set \[
S:=\bigcup_{i=1}^{N}D\left(\sum_{\substack{j=1\\
j\neq i}
}^{n}A_{ij},A_{ii}\right):=\bigcup_{i=1}^{N}S_{i}\,,\]
 where $D(\epsilon,\, x)\subseteq\mathbb{C}$ is the ball centred
at $x$ of radius $\epsilon$. Now, if $\lambda\in S_{i}$, then $\left|\lambda-A_{ii}\right|<\sum_{j\neq i}A_{ij}$,
so in particular \[
\Re(\lambda)>A_{ii}-\sum_{j\neq i}\left|A_{ij}\right|=A_{ii}+\sum_{j\neq i}A_{ij}=\sum_{j=1}^{n}A_{ij}>0\,.\]
 Now, all terms in the matrix $A$ are real, so if $\lambda$ is an
eigenvalue of $A$, then either $\lambda>0$, or $\bar{\lambda}$
is also an eigenvalue. We know that the determinant of a matrix is
given by the product of its eigenvalues, counting multiplicity. Therefore,
the determinant of $A$ is a product of positive reals, and terms
of the form $\lambda\bar{\lambda}=\left|\lambda\right|^{2}$, which
are also positive and real, so $\det(A)$ is positive, so $\det\left(A^{T}\right)$
is positive. 
\end{proof}
\begin{lem}
$V^{n}$ has only\label{lem:PosCoeffs} positive coefficients. 
\end{lem}
\begin{proof}
We define \[
d_{i}^{n}=\left|\begin{array}{c|c}
\overbrace{\begin{array}{ccccc}
- & + & - & \cdots & -\\
- & - & + & \ddots & \vdots\\
- & - & - & \ddots & -\\
\vdots & \vdots & \ddots & \ddots & +\\
- & \cdots & - & - & -\end{array}}^{i\times i} & -\\
\hline - & \begin{array}{ccccc}
+ & - & - & \cdots & -\\
- & + & - & \ddots & -\\
- & - & + & \ddots & \vdots\\
\vdots & \vdots & \ddots & \ddots & -\\
- & - & \cdots & - & +\end{array}\end{array}\right|\,.\]
 For our purposes, all that matters is the signs of the elements of
this matrix, and that Lemma \vref{lem:PosSquare} holds. Cyclically
permuting the first $i$ rows gives \[
d_{i}^{n}=(-1)^{i-1}\left|\begin{array}{ccccc}
- & - & - & - & -\\
- & + & - & - & -\\
- & - & + & - & \vdots\\
- & \vdots & \ddots & \ddots & -\\
- & - & \cdots & - & +\end{array}\right|=(-1)^{i-1}d_{1}^{n}\,.\]
 We can see that \begin{align*}
V^{n}= & \left|\begin{array}{cccccc}
\mathbf{e}_{0} & \mathbf{e}_{1} & \mathbf{e}_{2} & \mathbf{e}_{3} & \cdots & \mathbf{e}_{n}\\
- & + & - & - & \cdots & -\\
- & - & + & - & \cdots & -\\
- & - & - & + & \cdots & -\\
\vdots & \vdots & \vdots & \vdots & \ddots & \vdots\\
- & - & - & - & \cdots & +\end{array}\right|\\
= & (+)\mathbf{e}_{0}+\sum_{i=1}^{n}(-1)^{i}d_{i}^{n}\mathbf{e}_{i}\end{align*}
 and \begin{align*}
d_{1}^{n}= & (-(+))-(-d_{1}^{n-1})+(-d_{2}^{n-2})-\cdots+(-1)^{n-1}(-d_{n-1}^{n-1})\\
= & (-)+\sum_{j=1}^{n-1}(-1)^{j+1}(d_{j}^{n-1})\,.\end{align*}

We now proceed by induction. We first consider the case where $k=1$.
We can see that \[
\left|\begin{array}{cc}
\mathbf{e}_{0} & \mathbf{e}_{1}\\
- & +\end{array}\right|=(+)\mathbf{e}_{0}-(-)\mathbf{e}_{1}=(+)\mathbf{e}_{0}+(+)\mathbf{e}_{1}\,,\]
 so the lemma holds for $k=1$. Now suppose that the lemma holds for
$k-1$, and consider $V^{k}$. The $\mathbf{e}_{0}$ coordinate is
positive, by Lemma \vref{lem:PosSquare}. The $\mathbf{e}_{i}$ coordinate
is given by \begin{multline*}
(-1)^{i}d_{i}^{k}=(-1)^{i}(-1)^{i-1}d_{1}^{k}=(-)\left((-)+\sum_{j=1}^{k-1}(-1)^{j+1}(d_{j}^{k-1})\right)\\
=(+)+\sum_{j=1}^{k-1}\underbrace{(-1)^{j}(d_{j}^{k-1})}_{\mathbf{e}_{j}\text{term of }V^{k-1}}=(+)\,,\end{multline*}
 so the lemma holds for $k$, and so holds for all $k\in\mathbb{N}$. 
\end{proof}
\begin{cor}
F\label{cor:PosUnitKernel}or each $p\in\Pi$, the kernel of \[
M(p)=\left(\begin{array}{ccccc}
Q_{1}(p_{0}) & Q_{1}(p_{1}) & Q_{1}(p_{2}) & \cdots & Q_{1}(p_{n})\\
Q_{2}(p_{0}) & Q_{2}(p_{1}) & Q_{2}(p_{2}) & \cdots & Q_{2}(p_{n})\\
\vdots & \vdots & \vdots & \ddots & \vdots\\
Q_{n}(p_{0}) & Q_{n}(p_{1}) & Q_{n}(p_{2}) & \cdots & Q_{n}(p_{n})\end{array}\right)\]
 is one dimensional and spanned by a vector with strictly positive
entries. Further, we can define a continuous function $\kappa:\Pi\to\mathbb{R}^{n+1}$
such that $\kappa(p)$ is entry-wise positive, and $\kappa(p)$ is
in the kernel of $M(p)$. 
\end{cor}
\begin{proof}
We can see that $M(p)$ is always rank $n$, as the right hand $n\times n$
sub-matrix is invertible, by Lemma \ref{lem:PosSquare}, so its kernel
is everywhere rank one. If at each $p\in\Pi$ we take the $V^{n}$
defined earlier, and define this as $\kappa(p)$, it is clear that
this is entry-wise positive, orthogonal to the span of the row vectors
(so in the kernel of the operator), and has entries that sum to one,
from the definitions and the above proved theorems. 
\end{proof}

\subsection{Canonical Analytic Functions}

For $p\in\Pi$ we define \[
k_{p}=\sum_{j=0}^{n}\kappa_{j}(p)\mathbb{P}(\cdot,\, p_{j})\,,\]
 where $\kappa$ is as in corollary \ref{cor:PosUnitKernel}. Define
$\tau:\Pi\to\mathbb{R}^{n+1}$ by $\tau(p)=\kappa(p)/k_{p}(b)$. We
then define \[
h_{p}=\sum_{j=0}^{n}\tau_{j}(p)\mathbb{P}(\cdot,\, p_{j})\,.\]

It is clear that this corresponds to the measure \[
\mu=\sum_{j=0}^{n}\tau_{j}(p)\delta_{p_{j}}\]
 on $B$. We can see that $h_{p}$, thus defined, is a positive harmonic
function, with $h_{p}(b)=1$. We can also see that its periods are
zero, as \begin{multline}
P_{j}(h_{p})=\int_{B}Q_{j}d\mu=\int_{B}Q_{j}\sum_{i=0}^{n}\tau_{i}(p)\delta_{p_{i}}=\\
\sum_{i=0}^{n}\tau_{i}(p)\int_{B}Q_{j}\delta_{p_{i}}=\sum\tau_{i}(p)Q_{j}(p_{i})=0\,,\label{eq:NoPeriods}\end{multline}
 as $\tau(p)$ is in the kernel of $M(p)$, and \eqref{eq:NoPeriods}
is just the $j$-th coordinate of $M(p)\tau(p)$. The function $h_{p}$
is therefore the real part of an analytic function $f_{p}$ on $R$.
We require that $f_{p}(b)=1$.

We define $\mathcal{H}(R)$ as the space of holomorphic functions
on $R$, with the compact open topology. This is locally convex, metrisable,
and has the Heine-Borel property, that is, closed bounded subsets
of $\mathcal{H}(R)$ are compact. We then define \[
\mathbb{K}=\left\{ f\in\mathcal{H}(R):\, f(b)=1,\, f+\bar{f}>0\right\} \,.\]

\begin{lem}
The set $\mathbb{K}$ is compact. 
\end{lem}
\begin{proof}
$\mathbb{K}$ is clearly closed, so it suffices to show that $\mathbb{K}$
is bounded. The case where $R$ is the unit disc is proved in \cite{DritschelRationalDilation},
and we use this result without proof.

Since the $B_{0},\,\ldots,\, B_{n}$ are disjoint, closed sets, and
$R$ is $T_{4}$, we can find disjoint open sets $U_{0},\,\ldots,\, U_{n}$
containing each. By a simple topological argument we can show that
there exists some $E>0$ such that \[
O_{i}(E):=\left\{ z\in\mathbb{C}:\, d(z,B_{i})<E\right\} \subseteq U_{i}\,.\]
 It is clear that $R$ is covered by the family of connected compact
sets \[
\left\{ K_{\epsilon}\right\} :=\left\{ R\backslash\left(\bigcup_{i}O_{i}(\epsilon)\right):\,0<\epsilon<E\right\} \,,\]
 so it is sufficient to work with just these compact sets.

We choose a sequence of disjoint, simple paths $\upsilon_{0},\,\ldots,\,\upsilon_{n}$
through $X$ such that $\upsilon_{i}$ goes from $B_{i}$ to $B_{i+1}$,
and $\upsilon_{0}$ passes through $b$ (note that when $X$ is a
symmetric domain, $\upsilon_{i}=\mathbb{X}_{i}$ satisfies this).
It is clear that the union of these paths cuts $X$ into two disjoint,
simply connected sets $U$ and $V$. It is also possible to show that
we can choose a $\delta>0$ such that adding \[
W:=\{z\in R:\, d(z,\,\upsilon_{i})\leq\delta\text{ for some }i\}\]
 to either of these sets preserves simple connectivity. We can see
that $K_{\epsilon}^{+}:=K_{\epsilon}\cap(U\cup W)$ and $K_{\epsilon}^{-}:=K_{\epsilon}\cap(V\cup W)$
are simply connected compact sets containing $b$, whose union is
$K_{\epsilon}$. By the Riemann mapping theorem, we can canonically
map $K_{\epsilon}^{\pm}$ to the unit disc, in a way that takes $b$
to zero, so by the result of \cite{DritschelRationalDilation} mentioned
earlier, we have a constant $M_{\epsilon}^{\pm}$, such that $f$
analytic on $R$ with $f(b)=1$ implies for all $z\in K_{\epsilon}^{\pm}$,
$\vert f(z)\vert\leq M_{\epsilon}^{\pm}$. 
\end{proof}
\begin{lem}
The extreme points of $\mathbb{K}$ are precisely $\{f_{p}:\, p\in\Pi\}$. 
\end{lem}
\begin{proof}
Clearly, each $f_{p}$ is an extreme point of $\mathbb{K}$, so we
prove the converse -- if $f\neq f_{p}$, then $f$ is not an extreme
point of $\mathbb{K}$.

If $f\in\mathbb{K}$, then the real part of $f$ is a positive harmonic
function $h$ with $h(b)=1$. We therefore know that there is some
positive measure $\mu$ on $B$ such that \[
h(w)=\int_{B}\mathbb{P}(w,\, z)d\mu(z)\,.\]
 As $f$ is holomorphic, by Corollary \vref{cor:NotZeroOnBj}, $\mu$
must support at least one point on each $B_{i}$. If $f\neq f_{p}$,
then $\mu$ must support more than one point on some $B_{i}$.

Now, a note. We know $f$ is holomorphic if $P_{j}(h)=0$ for $j=0,\,\ldots,\, n$.
However, we know that $\sum_{j=0}^{n}Q_{j}=0$, so $\sum_{j=0}^{n}P_{j}(h)=0$,
so if we show that all but one of the $P_{j}(h)$ are zero, we have
shown that they are all zero, so $f$ is holomorphic.

With that in mind, suppose that $\mu$ supports more than one point
on $B_{0}$. We do not lose any generality by doing this, as relabelling
the boundary curves does not matter in the proof below, so we can
safely relabel any given boundary curve $B_{0}$. We divide $B_{0}$
into two parts, $A_{1}$ and $A_{2}$, in such a way that $\mu$ is
non-zero on both.

Now, let \[
a_{jl}=\int_{A_{l}}Q_{j}d\mu\,,\qquad l=1,\,2\,,\]
 and \[
k_{jm}=\int_{B_{m}}Q_{j}d\mu\,,\qquad m=1,\,\ldots,\, n\,,\]
 Since $h$ is the real part of an analytic function, \[
0=\int_{B}Q_{j}d\mu\,,\]
 so \[
\sum_{m=1}^{n}k_{jm}+a_{j1}+a_{j2}=0\,.\]
 Since $Q_{j}<0$ on $B_{i}$ for $i\neq j$, for any $M\subseteq\{1,\,\ldots,\, n\}$
containing $j$, \[
\sum_{m\in M}k_{jm}=-\left(a_{j1}+a_{j2}+\sum_{m\notin M}k_{jm}\right)>0\,.\]

We can now apply the Gershgorin circles trick from the proof of Lemma
\vref{lem:PosSquare}, to see that all sub-matrices of $K:=(k_{jm})$
of the form \[
\left(\begin{array}{ccccc}
+ & - & - & \cdots & -\\
- & + & - & \cdots & -\\
- & - & + & \cdots & -\\
\vdots & \vdots & \vdots & \ddots & \vdots\\
- & - & - & \cdots & +\end{array}\right)\]
 have positive determinant (including $K$, which must therefore be
invertible). We also note that the proof of Lemma \vref{lem:PosCoeffs}
only used this fact and the signs of the elements of matrices.

We consider the adjugate matrix $C$ of $K$, which is defined by
\[
c_{jm}=(-1)^{j+m}\left|\left(k_{\alpha\beta}\right)_{\substack{\alpha\neq j\\
\beta\neq m}
}\right|\]
 and has the property that $\det(K)^{-1}C^{T}=K^{-1}$. If we can
show that all the $c_{jm}$ are positive, then we will have that all
the entries of $K^{-1}$ are positive.

Now, if $j=m$, then \[
c_{jm}=\cancel{(-1)^{j+j}}\left|\begin{array}{ccccc}
+ & - & - & \cdots & -\\
- & + & - & \cdots & -\\
- & - & + & \cdots & -\\
\vdots & \vdots & \vdots & \ddots & \vdots\\
- & - & - & \cdots & +\end{array}\right|=(+)\,.\]
 If $m>j$ then $c_{jm}$ is given by \begin{equation}
(-1)^{j+m}\left|\begin{array}{c|c|c}
\begin{matrix}+ & - & \cdots & -\\
- & + & \cdots & -\\
\vdots & \vdots & \ddots & \vdots\\
- & - & \cdots & +\end{matrix} & - & -\\
\hline - & \overbrace{\begin{matrix}- & + & - & \cdots & -\\
- & - & + & \ddots & \vdots\\
- & - & - & \ddots & -\\
\vdots & \vdots & \ddots & \ddots & +\\
- & \cdots & - & - & -\end{matrix}}^{(m-j)\times(m-j)\text{ block}} & -\\
\hline - & - & \begin{matrix}+ & - & \cdots & -\\
- & + & \cdots & -\\
\vdots & \vdots & \ddots & \vdots\\
- & - & \cdots & +\end{matrix}\end{array}\right|\,.\label{eq:trans}\end{equation}
 By cyclically permuting the $m-j$ rows in the middle we get\[
(-1)^{\cancel{j-m}-1}\cancel{(-1)^{j+m}}\left|\begin{array}{cc|c|c}
 &  & {\scriptstyle \text{col }j}\\
 & \begin{array}{cccc}
+ & - & \cdots & -\\
- & + & \cdots & -\\
\vdots & \vdots & \ddots & \vdots\\
- & - & \cdots & +\end{array} & - & -\\
\hline {\scriptstyle \text{row }j} & - & - & -\\
\hline  & - & - & \begin{array}{cccc}
+ & - & \cdots & -\\
- & + & \cdots & -\\
\vdots & \vdots & \ddots & \vdots\\
- & - & \cdots & +\end{array}\end{array}\right|\,,\]
 and by cyclically permuting the first $j$ rows, and the first $j$
columns we get \[
(-1)\cancel{(-1)^{j-1}}\cancel{(-1)^{j-1}}d_{1}^{n-1}\,,\]
 which we note is precisely the $e_{1}$ term of $V^{n-1}$ in Lemma
\vref{lem:PosCoeffs}, which is positive.

If $j>m$, then $c_{jm}$ is given by\begin{equation}
(-1)^{j+m}\left|\begin{array}{c|c|c}
\begin{matrix}+ & - & \cdots & -\\
- & + & \cdots & -\\
\vdots & \vdots & \ddots & \vdots\\
- & - & \cdots & +\end{matrix} & - & -\\
\hline - & \overbrace{\begin{matrix}- & - & - & \cdots & -\\
+ & - & - & \ddots & \vdots\\
- & + & - & \ddots & -\\
\vdots & \vdots & \ddots & \ddots & -\\
- & \cdots & - & + & -\end{matrix}}^{(j-m)\times(j-m)\text{ block}} & -\\
\hline - & - & \begin{matrix}+ & - & \cdots & -\\
- & + & \cdots & -\\
\vdots & \vdots & \ddots & \vdots\\
- & - & \cdots & +\end{matrix}\end{array}\right|\,.\label{eq:untrans}\end{equation}
 But note that transposing matrices preserves determinant, and the
transpose of the matrix in \eqref{eq:untrans} is the matrix in \eqref{eq:trans},
so $c_{jm}=c_{mj}$, which we already know is positive. Therefore,
$K^{-1}$ has all positive entries. Since \[
\left(\begin{array}{c}
-a_{1l}\\
\vdots\\
-a_{nl}\end{array}\right)\]
 has all positive entries, we define \[
\left(\begin{array}{c}
b_{1l}\\
\vdots\\
b_{nl}\end{array}\right):=K^{-1}\left(\begin{array}{c}
-a_{1l}\\
\vdots\\
-a_{nl}\end{array}\right)\,.\]

Define positive measures $\nu_{1}$, $\nu_{2}$ by \[
\nu_{l}(A)=\mu(A\cap A_{l})+\sum_{m=1}^{n}b_{ml}\mu(A\cap B_{m})\,.\]
 Then \[
\int_{B}Q_{j}d\nu_{l}=a_{jl}+\sum_{m=1}^{n}k_{jm}b_{ml}=0\,,\]
 so each \[
h_{l}=\int_{B}\mathbb{P}(\cdot,\, w)d\nu_{l}(w)\,,\qquad l=1,\,2\,\,,\]
 is the real part of an analytic function $g_{l}$ with $\Im g_{l}(b)=0$.
We can see that $\nu_{1}+\nu_{2}=\mu$ as \[
K\left(\begin{array}{c}
1\\
\vdots\\
1\end{array}\right)=\left(\begin{array}{c}
\sum_{m=1}^{n}\int_{B_{m}}Q_{1}d\mu\\
\vdots\\
\sum_{m=1}^{n}\int_{B_{m}}Q_{n}d\mu\end{array}\right)=\left(\begin{array}{c}
\cancel{P_{1}(h)}-a_{11}-a_{12}\\
\vdots\\
\cancel{P_{n}(h)}-a_{n1}-a_{n2}\end{array}\right)\,.\]
 Multiplying both sides by $K^{-1}$ gives $b_{m1}+b_{m2}=1$. We
therefore have $h_{1}+h_{2}=h$. Thus, $g_{l}/g_{l}(b)\in\mathbb{K}$
and \[
f=g_{1}(b)\left(\frac{g_{1}}{g_{1}(b)}\right)+g_{2}(b)\left(\frac{g_{2}}{g_{2}(b)}\right)\,,\]
 so $f$ is a convex combination of two other points in $\mathbb{K}$.
Hence, $f$ is not an extreme point. 
\end{proof}
\begin{lem}
The set $\widehat{\mathbb{K}}$ of extreme points of $\mathbb{K}$
is a closed set, and the function taking $\Pi$ to $\mathbb{K}$ by
$p\mapsto f_{p}$ is a homeomorphism onto $\widehat{\mathbb{K}}$. 
\end{lem}
\begin{proof}
The proof is exactly as that of Lemma 2.11 in \cite{DritschelRationalDilation}. 
\end{proof}

\subsection{\label{sec:Test-Functions}Test Functions}

For $p\in\Pi$, define \[
\psi_{p}=\frac{f_{p}-1}{f_{p}+1}\,.\]
 The real part, $h_{p}$, of $f_{p}$ is harmonic across $B\backslash\{p_{0},\,\ldots,\, p_{n}\}$,
therefore $f_{p}$ is analytic across $B\backslash\{p_{0},\,\ldots,\, p_{n}\}$.
Also, $f_{p}$ looks locally like $g_{j}/(z-p_{j})$ at $p_{j}$,
for some analytic $g_{j}$, non-vanishing at $p_{j}$ (by \cite[Ch. 4, Prop. 6.4]{Fisher}).
We can see from this that $\psi_{p}$ is continuous onto $B$ and
$\left|\psi_{p}\right|=1$ on $B$.

By the reflection principle, $\psi_{p}$ is inner and extends analytically
across $B$, and $\psi_{p}^{-1}\{1\}=\{p_{0},\,\ldots,\, p_{n}\}$,
so the preimage of each point $z\in\mathbb{D}$ is exactly $n+1$
points, up to multiplicity, and so $\psi_{p}$ has $n+1$ zeroes.

Similarly, if $\psi$ is analytic in a neighbourhood of $R$, with
modulus one on $B$ and $n+1$ zeroes in $R$, then $\psi^{-1}\{1\}$
has $n+1$ points. Also, the real part of \[
f=\frac{1+\psi}{1-\psi}\]
 is a positive harmonic function which is zero on $B$ except where
$\psi(z)=1$. By Corollary \vref{cor:NotZeroOnBj}, $f$ cannot be
identically zero on any $B_{i}$, so there must be one point from
$\psi^{-1}\{1\}$ on each $B_{i}$. If, further, $\psi(b)=0$, then
$\psi=\psi_{p}$ for some $p\in\Pi$.

We define $\Theta=\left\{ \psi_{p}:\, p\in\Pi\right\} $.

\begin{thm}
\label{thm:Big}If $\rho$ is analytic in $R$ and if $\left|\rho\right|\leq1$
on $R$, then there exists a positive measure $\mu$ on $\Pi$ and
a measurable function $h$ defined on $\Pi$ whose values are functions
$h(\cdot,\, p)$ analytic in $R$ so that \[
1-\rho(z)\overline{\rho(w)}=\int_{\Pi}h(z,\, p)\left[1-\psi_{p}(z)\overline{\psi_{p}(w)}\right]\overline{h(w,\, p)}d\mu(p)\,.\]

\end{thm}
\begin{proof}
First suppose $\rho(b)=0$.

Let \[
f=\frac{1+\rho}{1-\rho}\]
 so \[
\rho=\frac{f-1}{f+1}\,.\]
 Hence \begin{equation}
1-\rho(z)\overline{\rho(w)}=2\frac{f(z)+\overline{f(w)}}{\left(f(z)+1\right)\left(\overline{f(w)}+1\right)}\label{eq:1-RhoRhoBar}\end{equation}

Since $h$, the real part of $f$, is positive and $f(b)=1$, the
function $f$ is in $\mathbb{K}$. Since $\mathbb{K}$ is a compact
convex subset of the locally convex topological vector space $\mathcal{H}(R)$,
by the Krein-Milman theorem, $f$ is in the closed convex hull of
$\widehat{\mathbb{K}}=\{f_{p}:\, p\in\Pi\}$, the set of extreme points
of $\mathbb{K}$. Therefore, there exists some regular Borel probability
measure $\nu$ on $\Pi$ such that \[
f=\int_{\Pi}f_{p}d\nu(p)\,.\]

Using the definition of $\psi_{p}$ and \eqref{eq:1-RhoRhoBar}, we
can show that \[
1-\rho(z)\overline{\rho(w)}=\int_{\Pi}\frac{1-\psi_{p}(z)\overline{\psi_{p}(w)}}{\left(f(z)+1\right)\left(1-\psi_{p}(z)\right)\left(1-\overline{\psi_{p}(w)}\right)\left(\overline{f(w)}+1\right)}d\nu(p)\,.\]

Finally, if $\rho(b)=a$, then we have a representation like the one
above, as \[
1-\left(\frac{\rho(z)-a}{1-\bar{a}\rho(z)}\right)\overline{\left(\frac{\rho(w)-a}{1-\bar{a}\rho(w)}\right)}=\frac{\left(1-a\bar{a}\right)\left(1-\rho(z)\overline{\rho(w)}\right)}{\left(1-\bar{a}\rho(z)\right)\left(1-a\overline{\rho(w)}\right)}\,.\]

\end{proof}
The interested reader may note that the set $\Theta$ is a collection
of test functions for $H^{\infty}(R)$, as defined in \cite{DritschelInterpolation}.

\begin{note}
\label{not:Minimal}We have used $n+1$ parameters to describe the
inner functions in $\Theta$, however, we only need $n$, as we can
identify them with the inner functions with $n+1$ zeroes, by the
argument in the introduction to Section \vref{sec:Test-Functions}.
If we then fix some $\tilde{p_{0}}\in B_{0}$, it is then clear that
for all $p\in\Pi$, $\overline{\psi_{p}(\tilde{p_{0}})}\psi_{p}$
is an inner function with $n+1$ zeroes, with one of them at $b$,
and $\overline{\psi_{p}(\tilde{p_{0}})}\psi_{p}(\tilde{p_{0}})=1$,
so $\overline{\psi_{p}(\tilde{p_{0}})}\psi_{p}=\psi_{q}$, where $q=(\tilde{p_{0}},\, q_{1},\,\ldots,\, q_{n})$,
for some $q_{1}\in B_{1},\,\ldots,\, q_{n}\in B_{n}$. We define \[
\widetilde{\Theta}:=\left\{ \psi_{q}:\, q=(\tilde{p_{0}},\, q_{1},\,\ldots,\, q_{n}),\, q_{1}\in B_{1},\,\ldots,\, q_{n}\in B_{n}\right\} \,,\]
 which is also a set of test functions for $H^{\infty}(R)$. 
\end{note}

\section{Matrix Inner Functions}

\subsection{Preliminaries}

\begin{thm}
\label{thm:OffX}If $R$ is symmetric, then there is some $b\in\mathbb{X}$,
and some $\psi_{\mathbf{p}}\in\widetilde{\Theta}$ with $n+1$ distinct
zeroes $b,\, z_{1,\,}\ldots,\, z_{n}$, where $z_{1},\,\ldots,\, z_{n}\notin\mathbb{X}$,
and $z_{i}\neq\varpi(z_{j})$ for all $i,\, j$. 
\end{thm}
\begin{proof}
For now, choose a $b_{0}\in R$, and use this as our $b$. We will
find a better choice for $b$ later in the proof. Take $p_{0}^{-}$
as $\tilde{p_{0}}$, and use this to define $\widetilde{\Theta}$
as in Note \vref{not:Minimal}. We will give this $\widetilde{\Theta}$
an unusual name, $\widetilde{\Theta}_{0}$, and call the functions
in it $\varphi_{p}$, rather than $\psi_{p}$. This is to distinguish
it from the $\widetilde{\Theta}$ and $\psi_{\mathbf{p}}$ in the
statement of the theorem, which we will construct later.

Choose some $p_{1}\in B_{1}\backslash\mathbb{X},\,\ldots,\, p_{n}\in B_{n}\backslash\mathbb{X}$.
Consider the path $\upsilon$ along $\mathbb{X}$ from $B_{1}$ to
$B_{0}$. Its image under $\varphi_{p}$ is a path leading to $1$.
We can see that $\varphi_{p}^{-1}\{1\}$ has $n+1$ points. As $X$
is Hausdorff and locally connected, there are disjoint, connected
open sets $U_{0},\, U_{1},\,\ldots,\, U_{n}$ around each of these
points, and since $\varphi_{p}$ is an open mapping on each of these
open sets, \[
\mathcal{N}:=\bigcap_{i=0}^{n}\varphi_{p}(U_{i})\]
 is a (relatively) open neighbourhood of $1$, whose preimage is $n+1$
disjoint open sets, $U_{0}^{\prime},\,\ldots,\, U_{n}^{\prime}$.
Also, we can choose $U_{1},\,\ldots,\, U_{n}$ such that none of them
intersects $\mathbb{X}$, and none of them intersects any $\varpi(U_{i})$
(since $p_{1},\,\ldots,\, p_{n}\notin\mathbb{X}$, and $\mathbb{X}$
closed). Now, we can lift $\varphi_{p}(\upsilon)\cap\mathcal{N}$
to each of these $U_{i}^{\prime}$, we choose a point $y\in\varphi_{p}(\upsilon)\cap\mathcal{N}$,
and note that $\varphi_{p}^{-1}\{y\}$ has exactly $n+1$ distinct
points, none of which maps to another under $\varpi$, and exactly
one of which is on $\mathbb{X}$. The point on $\mathbb{X}$, we use
as our $b$ for the rest of the proof. We take a Möbius transform
$m$ which preserves the unit circle, and maps $y$ to $0$, and notice
that $m\circ\varphi_{p}$ is an inner function which has $n+1$ zeroes,
exactly one of which, $b$, is on $\mathbb{X}$. If we define $\widetilde{\Theta}$
using our new $b$, and $\tilde{p}_{0}=p_{0}^{-}$, then $\overline{m\circ\varphi_{p}(p_{0}^{-})}m\circ\varphi_{p}\in\widetilde{\Theta}$,
and has the required zeroes, and so is our $\psi_{\mathbf{p}}$. 
\end{proof}
\begin{rem}
\label{rem:DistinctZeros}Note that in the above argument, we can
choose our $b$ as close to $p_{0}^{-}$ as we like, so in particular,
we can choose $b$ such that $h_{0}(b)>1/2$. By an argument similar
to that in \cite[Prop. 2.13]{DritschelRationalDilation}, we can see
that no $\psi_{p}\in\widetilde{\Theta}$ has all its zeroes at $b$. 
\end{rem}
\begin{thm}
If $R$ is symmetric, then $Q_{j}(p_{i})=\eta(p_{i})\, Q_{j}\left(\varpi(p_{i})\right)$,
for some $\eta:B\to\mathbb{C}$ which does not depend on $j$. 
\end{thm}
\begin{proof}
We write $Q_{j}$ as \[
Q_{j}(p)=\frac{\partial h_{j}}{\partial n_{p}}(p)\]
 where $\partial/\partial n_{p}$ is the normal derivative at $p$.
We also define $\partial/\partial t_{p}$ as the tangent derivative
at $p$.

Now, note that if $h$ is harmonic and $\varpi$ is anticonformal,
then $h\circ\varpi$ is also harmonic, and since $h_{j}$ and $h_{j}\circ\varpi$
have the same values on $B$, they must be equal, so \[
\frac{\partial h_{j}(p_{i})}{\partial n_{p_{i}}}=\frac{\partial h_{j}\left(\varpi(p_{i})\right)}{\partial n_{p_{i}}}\,,\]
 and so \begin{align*}
Q_{j}(p_{i})= & \frac{\partial h_{j}(p_{i})}{\partial n_{p_{i}}}\\
= & \frac{\partial h_{j}\left(\varpi(p_{i})\right)}{\partial n_{\varpi(p_{i})}}\cdot\frac{\partial n_{\varpi(p_{i})}}{\partial n_{p_{i}}}+\cancel{\frac{\partial h_{j}\left(\varpi(p_{i})\right)}{\partial t_{\varpi(p_{i})}}}\cdot\frac{\partial t_{\varpi(p_{i})}}{\partial n_{p_{i}}}\\
= & Q_{j}\left(\varpi(p_{i})\right)\cdot\underbrace{\frac{\partial n_{\varpi(p_{i})}}{\partial n_{p_{i}}}}_{\eta(p_{i})}\,.\end{align*}

\end{proof}
\begin{lem}
If $\eta$ is defined as\label{lem:InvolvedPoissonKernel} above,
and $b\in\mathbb{X}$ then \[
\mathbb{P}(b,\, p_{j})=\eta(p_{j})\,\mathbb{P}(b,\,\varpi(p_{j}))\,.\]

\end{lem}
\begin{proof}
We can write \[
\mathbb{P}(b,\, p_{j})=\frac{d\omega_{b}(p_{j})}{ds(p_{j})}\quad\text{and}\quad\mathbb{P}(b,\,\varpi(p_{j}))=\frac{d\omega_{b}(\varpi(p_{j}))}{ds(\varpi(p_{j}))}\,,\]
 and note that if $h$ is harmonic, then $h\circ\varpi$ is harmonic,
and $h\circ\varpi(b)=h(b)$. So, for any measurable set $E\subseteq B$,
\[
\omega_{b}(E)=\omega_{b}(\varpi(E))\,,\]
 so $d\omega_{b}(p_{j})=d\omega_{b}(\varpi(p_{j}))$. Hence, \begin{multline*}
\mathbb{P}(b,\, p_{j})=\frac{d\omega_{b}(p_{j})}{ds(p_{j})}=\frac{d\omega_{b}(\varpi(p_{j}))}{ds(p_{j})}=\frac{ds(\varpi(p_{j}))}{ds(p_{j})}\cdot\frac{d\omega_{b}(\varpi(p_{j}))}{ds(\varpi(p_{j}))}\\
=\frac{dn_{\varpi(p_{j})}}{dn_{p_{j}}}\cdot\mathbb{P}(b,\,\varpi(p_{j}))=\eta(p_{j})\,\mathbb{P}(b,\,\varpi(p_{j}))\,,\end{multline*}
 since \[
\frac{ds(\varpi(p_{j}))}{ds(p_{j})}\underbrace{=}_{\star}-\frac{dt_{\varpi(p_{j})}}{dt_{p_{j}}}\underbrace{=}_{\dagger}\frac{dn_{\varpi(p_{j})}}{dn_{p_{j}}}\,,\]
 where $\star$ is due to the fact that $\varpi$ is sense reversing,
and $\dagger$ is due to the Cauchy-Riemann equation for anti-holomorphic
maps. 
\end{proof}
\begin{defn}
We say a holomorphic $2\times2$ matrix valued function $F$ on $R$
has a \emph{standard zero set} if
\begin{enumerate}
\item $F$ has distinct zeroes $b,\, a_{1},\,\ldots,\, a_{2n}$, where $F(b)=0$,
and $\det\left(F\right)$ has zeroes of multiplicity one at each of
$a_{1},\,\ldots,\, a_{2n}$;
\item if $\gamma_{j}\neq0$ are such that $F(a_{j})^{*}\gamma_{j}=0$, $j=1,\,\ldots,\,2n$,
then no $n+1$ of the $\gamma_{j}$ lie on the same complex line through
the origin;
\item $Ja_{j}\neq P_{i}$ for $j=1,\,\ldots,\,2n$, $i=1,\,\ldots,\, n$,
where $P_{1},\,\ldots,\, P_{n}$ are the poles of the Fay kernel $K^{b}(\cdot,\, z)$.
\end{enumerate}
We have not defined $K^{b}$ yet, and will not do so until Section
\ref{sec:Theta-Functions}. For now, all we need to know about $K^{b}$
is that all its poles are on $J(\mathbb{X})$. 
\end{defn}

\subsection{The construction}

We take $\psi_{\mathbf{p}}$ as in Theorem \vref{thm:OffX}. Note
that $\overline{\psi_{\mathbf{p}}\circ\varpi}$ is an inner function
with zeroes at $b$, $\varpi(z_{1})$, $\ldots$, $\varpi(z_{n})$,
equal to one at $p_{0}^{-}$, $\varpi(\mathbf{p}_{1})$, $\varpi(\mathbf{p}_{2})$,
$\ldots$, $\varpi(\mathbf{p}_{n})$, so must equal $\psi_{\varpi(\mathbf{p})}$.

\begin{defn}
We say $S$ is a \emph{team of projections} if $S$ is a collection
of $n$ pairs of non-zero orthogonal projections on $\mathbb{C}^{2}$,
$\left(P^{j+},\, P^{j-}\right)$, such that \[
P^{1+}=\left(\begin{array}{cc}
1 & 0\\
0 & 0\end{array}\right)\,,\quad P^{1-}=\left(\begin{array}{cc}
0 & 0\\
0 & 1\end{array}\right)\,,\quad P^{j+}+P^{j-}=I\,,\quad j=1,\,\ldots n\,.\]
 Let $S_{0}$ be the \emph{trivial team}, given by $P^{j\pm}=P^{1\pm}$
for all $j$.

We define \[
H_{S,p}=\tau_{0}(p)\mathbb{P}(\cdot,\, p_{0}^{-})I+\sum_{i=1}^{n}\tau_{i}(p)\left[\mathbb{P}(\cdot,\, p_{i})\, P^{i+}+\eta(p_{i})\mathbb{P}\left(\cdot,\,\varpi(p_{i})\right)\, P^{i-}\right]\,.\]

We note that, by Lemma \ref{lem:InvolvedPoissonKernel},\begin{align*}
H_{S,p}(b)= & \tau_{0}(p)\mathbb{P}(b,\, p_{0}^{-})I+\sum_{i=1}^{n}\tau_{i}(p)\left[\mathbb{P}(b,\, p_{i})\, I\right]\\
= & \left[\sum_{i=0}^{n}\tau_{i}(p)\,\mathbb{P}(b,\, p_{i})\right]\, I=\cancel{h_{p}(b)}\, I=I\,.\end{align*}

For $x\in\mathbb{C}^{2}$ a unit vector, $\left\langle H_{S,p}\, x,\, x\right\rangle $
corresponds to the measure \[
\mu_{x,x}=\tau_{0}\delta_{p_{0}^{-}}+\sum_{i=1}^{n}\tau_{i}\cdot\left[\delta_{p_{i}}\,\left\Vert P^{i+}x\right\Vert +\delta_{\varpi(p_{i})}\eta(p_{i})\,\left\Vert P^{i-}x\right\Vert \right]\,,\]
 so \begin{align*}
\int_{B}Q_{j}d\mu_{x,x} & =\tau_{0}Q_{j}(p_{0}^{-})+\sum_{i=1}^{n}\tau_{i}\left[Q_{j}(p_{i})\,\left\Vert P^{i+}x\right\Vert +\eta(p_{i})Q_{j}\left(\varpi(p_{i})\right)\,\left\Vert P^{i-}x\right\Vert \right]\\
 & =\tau_{0}Q_{j}(p_{0}^{-})+\sum_{i=1}^{n}\tau_{i}Q_{j}(p_{i})\,\cancel{\left\Vert x\right\Vert }\\
 & =0\,,\end{align*}
 by definition of $\tau$.

Hence, $\left\langle H_{S,p}\, x,\, x\right\rangle $ is the real
part of an analytic function, so $H_{S,p}$ is the real part of a
holomorphic $2\times2$ matrix function $G_{S,p}$, normalised by
$G_{S,p}(b)=I$.

We now define \[
\Psi_{S,p}=\left(G_{S,p}-I\right)\cdot\left(G_{S,p}+I\right)^{-1}\,.\]

\end{defn}
\begin{lem}
\label{lem:PsiSp}If $\mathbf{p}$ is as in Theorem \vref{thm:OffX},
for each $S$:
\begin{enumerate}
\item $\Psi_{S,\mathbf{p}}$ is analytic in a neighbourhood of $X$ and
unitary v\label{itm:PsiAnalyticUnitary}alued on $B$;
\item $\Psi_{S,\mathbf{p}}(b)=0$\label{itm:Psi(b)=00003D00003D0};
\item $\Psi_{S,\mathbf{p}}(p_{0}^{-})=I$\label{itm:Psi=00003D00003D1};
\item $\Psi_{S,\mathbf{p}}(\mathbf{p}_{1})e_{1}=e_{1}$ and $\Psi_{S,\mathbf{p}}\left(\varpi(\mathbf{p}_{1})\right)e_{2}=e_{2}$\label{itm:Psi=00003D00003De1};
\item $\Psi_{S,\mathbf{p}}(\mathbf{p}_{i})\, P^{i+}=P^{i+}$ and $\Psi_{S,\mathbf{p}}(\varpi(\mathbf{p}_{i}))\, P^{i-}=P^{i-}$\label{itm:Psi=00003D00003DP1};
\item $\Psi_{S_{0},\mathbf{p}}=\left(\begin{smallmatrix}\psi_{\mathbf{p}} & 0\\
0 & \psi_{\varpi(\mathbf{p})}\end{smallmatrix}\right)$\label{itm:Psi0}.
\end{enumerate}
\end{lem}
\begin{proof}
Thinking about $\mathbb{P}(z,\, r)$ as a function of $z$, in a neighbourhood
of $r\in B$, the Poisson kernel $\mathbb{P}(z,\, r)$ is the real
part of some function of the form $g_{r}(z)(z-r)^{-1}$, where $g_{r}$
is analytic in the neighbourhood, and non-vanishing at $r$ (by \cite[Ch. 4, Prop. 6.4]{Fisher}).
At any other point $q\in B$, $\mathbb{P}(z,\, r)$ extends to a harmonic
function on a neighbourhood of $q$, so must be the real part of some
analytic function, with real part $0$ at $q$.

We can see that if $r\in B$ is not $p_{0}^{-},\,\mathbf{p}_{1},\,\ldots,\,\mathbf{p}_{n},\,\varpi(\mathbf{p}_{1}),\,\ldots,\,\varpi(\mathbf{p}_{n})$,
then $G_{S,\mathbf{p}}$ is analytic in a neighbourhood of $r$. Further,
$G_{S,\mathbf{p}}+I$ is invertible near $r$ as $G_{S,\mathbf{p}}(z)=H_{S,\mathbf{p}}(z)+iA(z)$
for some self-adjoint matrix valued function $A(z)$, and $H_{S,\mathbf{p}}(r)=0$.
Thus, $G_{S,\mathbf{p}}$ is invertible at and, by continuity, near
$r$. We have \[
I-\Psi_{S,\mathbf{p}}\Psi_{S,\mathbf{p}}^{*}=2(G_{S,\mathbf{p}}+I)^{-1}\underbrace{(G_{S,\mathbf{p}}+G_{S,\mathbf{p}}^{*})}_{iA+(iA)^{*}=0}(G_{S,\mathbf{p}}+I)^{*-1}\,,\]
 which is zero at $r$, so $\Psi_{S,\mathbf{p}}$ must be unitary
at $r$.

From the definition of $G_{S,\mathbf{p}}$, in a neighbourhood of
$p_{0}^{-}$, there are analytic functions $g_{1},\, g_{2},\, k_{1},\, k_{2}$
so that the real parts of $k_{j}$ are $0$ at $p_{0}^{-}$, each
$g_{j}$ is non-vanishing at $p_{0}^{-}$, and \[
G_{S,\mathbf{p}}(z)=\left(\begin{array}{cc}
\frac{g_{1}(z)}{z-p_{0}^{-}} & k_{1}(z)\\
k_{2}(z) & \frac{g_{2}(z)}{z-p_{0}^{-}}\end{array}\right)\,,\]
 so \begin{align*}
(G_{S,\mathbf{p}}(z)+I)^{-1}= & \frac{1}{\frac{g_{1}+z-p_{0}^{-}}{z-p_{0}^{-}}\frac{g_{2}+z-p_{0}^{-}}{z-p_{0}^{-}}-k_{1}(z)k_{2}(z)}\left(\begin{array}{cc}
\frac{g_{2}(z)-z-p_{0}^{-}}{z-p_{0}^{-}} & -k_{1}(z)\\
-k_{2}(z) & \frac{g_{1}(z)-z-p_{0}^{-}}{z-p_{0}^{-}}\end{array}\right)\\
= & \frac{\left(\begin{array}{cc}
\left(g_{2}(z)-z-p_{0}^{-}\right)(z-p_{0}^{-}) & -k_{1}(z)(z-p_{0}^{-})^{2}\\
-k_{2}(z)(z-p_{0}^{-})^{2} & \left(g_{1}(z)-z-p_{0}^{-}\right)(z-p_{0}^{-})\end{array}\right)}{\left(g_{1}(z)-z-p_{0}^{-}\right)\left(g_{2}(z)-z-p_{0}^{-}\right)-k_{1}(z)k_{2}(z)\,(z-p_{0}^{-})^{2}}\,.\end{align*}
 Note that the denominator is non-zero at and near $p_{0}^{-}$, so
$G_{S,\mathbf{p}}+I$ is invertible. We can use this to calculate
$\Psi_{S,\mathbf{p}}$ directly%
\footnote{The calculation is omitted, but can be readily verified by hand, or
with a computer algebra system%
}, and show that $\Psi_{S,\mathbf{p}}$ is analytic in a neighbourhood
of $p_{0}^{-}$, and $\Psi_{S,\mathbf{p}}(p_{0}^{-})=I$, so we have
(\ref{itm:Psi=00003D00003D1}).

Now we look at $\mathbf{p}_{1}$. Near $\mathbf{p}_{1}$ we have analytic
functions $g,\, k_{1},\, k_{2},\, k_{3}$, on a neighbourhood of $\mathbf{p}_{1}$,
where $k_{1},\, k_{2},\, k_{3}$ have zero real part at $\mathbf{p}_{1}$,
$g$ is non-zero at $\mathbf{p}_{1}$, and \[
G_{S,\mathbf{p}}(z)=\left(\begin{array}{cc}
\frac{g(z)}{z-\mathbf{p}_{1}} & k_{1}\\
k_{2} & k_{3}\end{array}\right)\,.\]
 Since $k_{3}+1$ has real part $1$ at $\mathbf{p}_{1}$, $g(z)\,(z-\mathbf{p}_{1})^{-1}$
has a pole, and $k_{1},\, k_{2}$ are analytic at $\mathbf{p}_{1}$,
we see that $G_{S,\mathbf{p}}$ is invertible near $\mathbf{p}_{1}$.
By direct computation, we see that $\Psi_{S,\mathbf{p}}$ is analytic
in a neighbourhood of $\mathbf{p}_{1}$ and \[
\Psi_{S,\mathbf{p}}(\mathbf{p}_{1})=\left(\begin{array}{cc}
1 & 0\\
0 & \frac{k_{3}(\mathbf{p}_{1})-1}{k_{3}(\mathbf{p}_{1})+1}\end{array}\right)\,.\]
 A similar argument holds for $\varpi(\mathbf{p}_{1})$, so we have
(\ref{itm:Psi=00003D00003De1}), and by working in the orthonormal
basis induced by $P^{j+}$ and $P^{j-}$, (\ref{itm:Psi=00003D00003DP1})
follows. Also, we have now shown $\Psi_{S,\mathbf{p}}$ is analytic
at every point, so (\ref{itm:PsiAnalyticUnitary}) follows.

(\ref{itm:Psi0}) and (\ref{itm:Psi(b)=00003D00003D0}) follow easily
from the definitions. 
\end{proof}
\begin{lem}
\label{lem:PsiSmp}We define $\left\Vert S_{1}-S_{2}\right\Vert _{\infty}=\max_{j\pm}\left\Vert P_{1}^{j\pm}-P_{2}^{j\pm}\right\Vert $,
giving a metric on the space $\mathcal{T}$ of all teams of projections.
There exists some non-trivial sequence $S_{m}\to S_{0}$ such that
for all $m$, $\Psi_{S_{m},\mathbf{p}}$ has a standard zero set. 
\end{lem}
\begin{proof}
Since the zeroes of $\psi_{\mathbf{p}}$ and $\psi_{\varpi(\mathbf{p})}$
are all distinct except for $b$, it is clear that \[
\Psi_{S_{0},\mathbf{p}}=\left(\begin{array}{cc}
\psi_{\mathbf{p}} & 0\\
0 & \psi_{\varpi(\mathbf{p})}\end{array}\right)\]
 has a standard zero set.

We note that whatever value we take for $\epsilon$, there is an $S\neq S_{0}$
within $\epsilon$ of $S_{0}$, so there is some non-trivial sequence
$S_{m}$ converging to $S_{0}$.

The sequence $\Psi_{S_{m},\mathbf{p}}$ is uniformly bounded, so has
a sub-sequence $\Psi_{m}$ which converges uniformly on compact subsets
of $R$ to some $\Psi$. This means \[
G_{m}=(I+\Psi_{m})(I-\Psi_{m})^{-1}\]
 converges uniformly on compact subsets of $R$ to \[
G=(I+\Psi)(I-\Psi)^{-1}\,.\]
 $H_{m}$, the real part of $G_{m}$ is harmonic, and \[
H_{m}-H_{0}=\sum_{i=2}^{n}\tau_{i}(\mathbf{p})\mathbb{P}(\cdot,\,\mathbf{p}_{i})\left[P_{m}^{i+}-P^{1+}\right]+\tau_{i}\left(\varpi(\mathbf{p})\right)\mathbb{P}(\cdot,\,\varpi(\mathbf{p}_{i}))\left[P_{m}^{i-}-P^{1-}\right]\,.\]
 Since $P_{m}^{i\pm}\rightarrow P^{1\pm}$, we see that $H_{m}\to H_{0}$,
and since $G(b)=I=G_{0}(b)$, $G_{m}\to G_{0}$, so $\Psi=\Psi_{0}$,
and $\Psi_{m}\to\Psi_{0}$ uniformly on compact sets.

Let $d_{m}(z)=\det(\Psi_{m}(z))$. This is analytic, and unimodular
on $B$. Draw small, disjoint circles in $R$ around the zeroes of
$d_{0}$ (which correspond to the zeroes of $\Psi_{0}$). By Hurwitz's
theorem, there exists some $M$ such that for all $m\geq M$, $d_{m}$
and $d_{0}$ have the same number of zeroes in each of these circles,
so the zeroes of $d_{m}$ must be distinct, apart from the repeated
zero at $b$. In particular, the zeroes ($b,\, a_{1}^{m},\,\ldots,\, a_{2n}^{m}$)
of $\Psi_{m}$ converge to the zeroes ($b,\, a_{1}^{0},\,\ldots,\, a_{2n}^{0}$)
of $\Psi_{0}$.

Finally, if $\Vert\gamma_{1}^{m}\Vert=1$, $\Psi_{m}(a_{1}^{m})^{*}\gamma_{1}^{m}=0$
and $a_{1}^{m}$ is close to $a_{1}^{0}$, then \[
\Psi_{0}(a_{1}^{0})^{*}\gamma_{1}^{m}=\left(\Psi_{0}(a_{1}^{0})-\Psi_{0}(a_{1}^{m})\right)^{*}\gamma_{1}^{m}+\left(\Psi_{0}(a_{1}^{m})-\Psi_{m}(a_{1}^{m})\right)^{*}\gamma_{1}^{m}\,.\]
 However, the right hand side tends to zero as $m$ tends to infinity,
so the projection of $\gamma_{1}^{m}$ onto the image of $\Psi_{0}(a_{1}^{0})^{*}$
tends to zero. Since $\gamma_{1}^{m}$ is a bounded sequence in a
finite-dimensional complex space, it has a convergent sub-sequence,
which we shall also call $\gamma_{1}^{m}$. This $\gamma_{1}^{m}$
must converge to something in the kernel of $\Psi_{0}(a_{1}^{0})^{*}$,
that is, a multiple of $e_{1}$. We apply this argument to $a_{2},\,\ldots,\, a_{2n}$,
and find a sub-sequence such that $n$ of the $\gamma_{i}^{m}$s tend
to multiples of $e_{1}$ and $n$ of them tend to multiples of $e_{2}$,
so for $m$ big enough, no $n+1$ of them are collinear. 
\end{proof}

\section{\label{sec:Theta-Functions}Theta Functions}

\subsection{The Jacobian Variety}

We know that for each $i=1,\,\ldots,\, n$, $h_{i}$ is locally the
real part of an analytic function $g_{i}$. The differential $dg_{i}$
can be extended from $R$ to $Y$ (as in Theorem \ref{thm:RedGreen},
$Y$ is the Schottky double of $R$), and \[
\alpha_{i}:=\frac{1}{2}dg_{i}\,,\qquad i=1,\,\ldots,\, n\]
 is then a basis for the space of holomorphic 1-forms on $Y$. We
see that if we define a homology basis for $Y$ by $A_{j}=\mathbb{X}_{j}-J(\mathbb{X}_{j})$
and $B_{j}$ as before, then $\int_{A_{j}}\alpha_{i}=\delta_{ij}$
and \[
\Omega:=\left(\int_{B_{j}}\alpha_{i}\right)_{ij}\]
 has positive definite imaginary part (see, for example, \cite[III.2.8]{FarkasKra}).

We define a lattice \[
L:=\mathbb{Z}^{n}+\Omega\mathbb{Z}^{n}\subseteq\mathbb{C}^{n}\]
 define the \emph{Jacobian variety} by \[
\mathcal{J}(Y):=\mathbb{C}^{n}/L\,,\]
 and define the \emph{Abel-Jacobi maps} $\chi:Y\to\mathbb{C}^{n}$
and $\chi_{0}:Y\to\mathcal{J}(Y)$ by \[
\chi(y):=\left(\begin{array}{c}
\int_{p_{0}^{-}}^{y}\alpha_{1}\\
\vdots\\
\int_{p_{0}^{-}}^{y}\alpha_{n}\end{array}\right)\,,\qquad\chi_{0}(y)=\left[\chi(y)\right]\,.\]
 Note that the integral depends on the path integrated over. However,
any two paths differ only by a closed path, and $A_{1},\,\ldots,\, A_{n},\, B_{1},\,\ldots,\, B_{n}$
is a homology basis for $Y$, so any closed path is homologous to
a sum of paths in this basis. Also, \[
\int_{A_{j}}\alpha_{i}\,,\:\int_{B_{j}}\alpha_{i}\in L\text{, so }\left[\int_{A_{j}}\alpha_{i}\right]=\left[\int_{B_{j}}\alpha_{i}\right]=0\,,\]
 so the choice of path to integrate over does not affect $\chi_{0}(y)$.

\begin{prop}
The Abel-Jacobi map has the following properties:
\begin{enumerate}
\item $\chi_{0}$ is a one-one\label{itm:into} conformal map of $Y$ onto
its image in $\mathcal{J}(Y)$; and
\item $\chi_{0}(Jy)=-\chi_{0}(y)^{*}$, where $^{*}$ denotes the coordinate-wise
conjugate\label{itm:mirror}.
\end{enumerate}
\end{prop}
\begin{proof}
(\ref{itm:into}) is proved in \cite[III.6.1]{FarkasKra}, (\ref{itm:mirror})
holds because $p_{0}^{-}\in\mathbb{X}$ and \[
g_{j}(Jy)-g_{j}(p_{0}^{-})=-\overline{\left(g_{j}(y)-g_{j}(p_{0}^{-})\right)}\,.\]

\end{proof}

\subsection{Theta Functions}

\begin{defn}
Roughly following \cite{MumfordI}, we define the \emph{theta function}
$\vartheta:\mathbb{C}^{n}\to\mathbb{C}$ by \[
\vartheta(z)=\sum_{m\in\mathbb{Z}^{n}}\exp\left(\pi i\left\langle \Omega m,\, m\right\rangle +2\pi i\left\langle z,\, m\right\rangle \right)\,,\]
 where $\left\langle \cdot,\cdot\right\rangle $ is the usual $\mathbb{C}^{n}$
inner product. This function is quasi-periodic, as \begin{align*}
\vartheta(z+m)= & \vartheta(z)\\
\vartheta(z+\Omega m)= & \exp\left(-\pi i\left\langle \Omega m,\, m\right\rangle -2\pi i\left\langle z,\, m\right\rangle \right)\vartheta(z)\end{align*}
 for all $m\in\mathbb{Z}^{n}$, as shown in \cite{MumfordI}. Given
$e\in\mathbb{C}^{n}$, we rewrite this as $e=u+\Omega v$ for some
$u,\, v\in\mathbb{R}^{n}$, and we define the \emph{theta function
with characteristic $e$}, $\vartheta[e]:\mathbb{C}^{n}\to\mathbb{C}$
by \[
\vartheta[e](z)=\vartheta\left[\begin{array}{c}
u\\
v\end{array}\right](z)=\exp\left(\pi i\left\langle \Omega v,\, v\right\rangle +2\pi i\left\langle z+u,\, v\right\rangle \right)\vartheta(z+e)\,.\]
 Note that this follows \cite{MumfordI}. Subtly different definitions
are used in \cite{Fay}, \cite{DritschelRationalDilation} and \cite{FarkasKra},
although these differences are not particularly important. 
\end{defn}
\begin{thm}
There exists a constant vector $\Delta$, depending on the choice
of base-point, such that for each $e\in\mathbb{C}^{n}$, either $\vartheta[e]\circ\chi$
is identically zero, or $\vartheta[e]\circ\chi$ has exactly $n$
zeroes, $\zeta_{1},\,\ldots,\,\zeta_{n}$ and \[
\sum_{i=1}^{n}\chi(\zeta_{i})=\Delta-e\,.\]

\end{thm}
\begin{proof}
See \cite[Ch. 2, Cor. 3.6]{MumfordI} or \cite[VI.2.4]{FarkasKra}. 
\end{proof}
For the following, it will be convenient to define \[
\mathcal{E}_{e}(x,\, y)=\vartheta(\chi(y)-\chi(x)+e)\,.\]

\begin{thm}
If $e\in\mathbb{C}^{n}$, $\vartheta(e)=0$ and $\mathcal{E}_{e}$
is not identically zero, then there exist $\zeta_{1},\,\ldots,\zeta_{n-1}$
such that for each $x\in Y$, $x\neq\zeta_{i}$, the\label{thm:PrimeZeros}
zeroes of $\vartheta[e-\chi(x)]\circ\chi$, which coincide with the
zeroes of $\mathcal{E}_{e}(x,\,\cdot)$, are precisely $x,\,\zeta_{1},\,\ldots,\,\zeta_{n-1}$. 
\end{thm}
\begin{proof}
See \cite[Ch. 2, Lemma 3.4]{MumfordI}. 
\end{proof}
\begin{thm}
There\label{thm:OddHalfPeriod} exists an $e_{*}=u_{*}+\Omega v_{*}\in\mathbb{C}^{n}$
such that $2e_{*}=0\text{ mod }L$, $\left\langle u_{*},\, v_{*}\right\rangle $
is an odd integer, and $\mathcal{E}_{e_{*}}\not\equiv0$. 
\end{thm}
For the proof see \cite[Ch. IIIb, Sec. 1, Lemma 1]{MumfordII}, although
the remarks at the end of \cite[VI.1.5]{FarkasKra} provide some relevant
discussion. An $e_{*}$ of this type is called a \emph{non-singular
odd half-period}, and we see that $\vartheta[e_{*}]$ is an odd function,
so $\vartheta(e_{*})=0$.

Let $\vartheta_{*}:=\vartheta[e_{*}]$, so \[
\vartheta_{*}(t)=\exp\left(\pi i\left\langle \Omega v_{*},\, v_{*}\right\rangle +2\pi i\left\langle z+u_{*},\, v_{*}\right\rangle \right)\vartheta(z+e_{*})\,.\]

Clearly, we can apply Theorems \ref{thm:PrimeZeros} and \ref{thm:OddHalfPeriod},
and get that the roots of \[
\vartheta_{*}\left(\chi(\cdot)-\chi(z)\right)\]
 are $\left\{ z,\,\zeta_{1},\,\ldots,\,\zeta_{n-1}\right\} $ for
some $\zeta_{1},\,\ldots,\,\zeta_{n-1}$. If neither of $z,\, w\in Y$
coincide with with any of these $\zeta_{i}$s, then \begin{equation}
\frac{\vartheta_{*}\left(\chi(\cdot)-\chi(z)\right)}{\vartheta_{*}\left(\chi(\cdot)-\chi(w)\right)}=e^{2\pi i\left\langle w-z,v_{*}\right\rangle }\frac{\vartheta\left(\chi(\cdot)-\chi(z)+e_{*}\right)}{\vartheta\left(\chi(\cdot)-\chi(w)+e_{*}\right)}\label{eq:PrimeRatio}\end{equation}
 is a multiple valued function with exactly one zero and one pole,
at $z$ and $w$ respectively.

\subsection{The Fay Kernel}

A tool that will prove invaluable in later sections is the Fay kernel
$K^{a}$, which is a reproducing kernel on $\mathbb{H}^{2}(R,\,\omega_{a})$,
the Hardy space of analytic functions on $R$ with boundary values
in $L^{2}(\omega_{a})$. For a more comprehensive discussion of the
ideas in this section, see \cite{Fay}.

\begin{lem}
The critical points of the Green's function $g(\cdot,\, b)$ are on
$\mathbb{X}$, one in each $\mathbb{X}_{i}$, $i=1,\,\ldots,\, n$. 
\end{lem}
\begin{proof}
We write $g(z)=g(z,\, b)$. We know that $g$ has $n$ critical points,
by \cite[p. 133-135]{NehariConformal}. On $\mathbb{X}$, define $\partial/\partial x$
as the derivative tangent to $\mathbb{X}$, and $\partial/\partial y$
as the derivative normal to $\mathbb{X}$. We know that $g\circ\varpi=g$,
and \[
\frac{\partial g}{\partial y}=\frac{\partial g\circ\varpi}{\partial y}=\frac{\partial g}{\partial y}\cdot\frac{\partial\varpi_{y}}{\partial y}+\frac{\partial g}{\partial x}\cdot\cancel{\frac{\partial\varpi_{x}}{\partial y}}.\]
 However, $\partial\varpi_{y}/\partial y<0$ on $\mathbb{X}$, so
the two sides of this equation have different signs, and so $\partial g/\partial y=0$
on $\mathbb{X}$. Also, $g=0$ on $B$, so $g$ must be zero at $p_{i}^{-}$
and $p_{i+1}^{+}$ -- the start and end points of $\mathbb{X}_{i}$.
Since $\partial g/\partial x$ is continuous on $\mathbb{X}_{i}$
(provided $i\neq0$), $\partial g/\partial x$ must be zero somewhere
on $\mathbb{X}_{i}$, by Rolle's theorem. Since this gives us $n$
distinct zeroes, this must be all of them. 
\end{proof}
We have just proved that $g(\cdot,\, b)$ has $n$ distinct zeroes.
If these zeroes are $z_{1}\in\mathbb{X}_{1},\,\ldots,\, z_{n}\in\mathbb{X}_{n}$
we define $P_{i}=Jz_{i}$.

\begin{thm}
\label{thm:FayKernel}There is a reproducing kernel $K^{a}$ for the
Hardy space $\mathbb{H}^{2}(R,\,\omega_{a})$; that is, if $f\in\mathbb{H}^{2}(R,\,\omega_{a})$,
then \[
f(y)=\left\langle f(\cdot),\, K^{a}(\cdot,\, y)\right\rangle =\int_{\partial R}f(x)\overline{K^{a}(x,\, y)}d\omega_{a}(x)\,.\]
 If $z=a$, then $K^{a}(\cdot,\, z)\equiv1$. If not, $K^{a}(\cdot,\, z)$
has precisely the poles \[
P_{1}(a),\,\ldots,\, P_{n}(a),\, Jz\]
 (where $JP_{1}(a),\,\ldots,\, JP_{n}(a)$ are the critical points
of $g(\cdot,\, a)$), and $n+1$ zeroes in $Y$, one of which is $Ja$. 
\end{thm}
\begin{proof}
{[}Sketch Proof]By \cite[Prop. 6.15]{Fay}, there is an $e\in\mathcal{J}(Y)$
such that \begin{multline}
K^{a}(x,\, y)=\\
\frac{\vartheta\left(\chi(x)+\chi(y)^{*}+e\right)\vartheta\left(\chi(a)+\chi(a)^{*}+e\right)\vartheta_{*}\left(\chi(a)+\chi(y)^{*}\right)\vartheta_{*}\left(\chi(x)+\chi(a)^{*}\right)}{\vartheta\left(\chi(a)+\chi(y)^{*}+e\right)\vartheta\left(\chi(x)+\chi(a)^{*}+e\right)\vartheta_{*}\left(\chi(x)+\chi(y)^{*}\right)\vartheta_{*}\left(\chi(a)+\chi(a)^{*}\right)}\label{eq:Fay}\end{multline}
 is the reproducing kernel%
\footnote{Fay gives \eqref{eq:Fay} in a slightly different form, although we
can use \cite[Prop. 6.1]{Fay} and some basic results on theta functions
to show that the two forms are equivalent. Also note that the notation
Fay uses differs significantly from the notation used here.%
} for $\mathbb{H}^{2}(R,\,\omega_{a})$.

It is clear that $K^{a}(x,\, a)=1$, so we fix $a$ and $y$ and look
at the zero/pole structure of $K^{a}(\cdot,\, y)$. We can see that
for fixed $y$, the zeroes and poles of \eqref{eq:Fay} are precisely
the zeroes and poles of \[
\frac{\vartheta\left(\chi(x)+\chi(y)^{*}+e\right)\vartheta_{*}\left(\chi(x)+\chi(a)^{*}\right)}{\vartheta\left(\chi(x)+\chi(a)^{*}+e\right)\vartheta_{*}\left(\chi(x)+\chi(y)^{*}\right)}\,,\]
 by removing terms with no dependence on $x$. By \eqref{eq:PrimeRatio},
the $\vartheta_{*}$ factors bring in a zero at $Ja$ and a pole at
$Jy$. The remaining theta functions have $n$ zeroes each, so $K^{a}$
gets $n$ new poles, $P_{1}(a),\,\ldots,\, P_{n}(a)$, and $n$ new
zeroes, $Z_{1}(y),\,\ldots,\, Z_{n}(y)$ from the top and bottom terms
respectively. The $P_{i}(a)$s must all be in $J(R)\cup B$, as we
know that $K^{a}(\cdot,\, y)$ is analytic on $R$.

Suppose, towards a contradiction, that some of these poles and zeroes
were to cancel, then $K^{a}(\cdot,\, y)$ would have $n$ or fewer
poles. If it had no zeroes, it would be constant, but we know that
the set $\left\{ K^{a}(\cdot,\, y):\, y\in R\right\} $ is linearly
independent, and $K^{a}(\cdot,\, a)$ is constant, so $K^{a}(\cdot,\, y)$
cannot be a multiple of it. If it had one or more poles, then it would
be a meromorphic function on $Y$ with between $1$ and $n$ poles,
all in $J(R)\cup B$. Moreover, $Jy$ cannot cancel with $Ja$ because
$a\neq y$, and it cannot cancel with any of the $Z_{i}(y)$s since
that would mean \begin{align*}
0= & \vartheta\left(\chi(Jy)+\chi(y)^{*}+e\right)\\
= & \vartheta\left(\cancel{-\chi(y)^{*}+\chi(y)^{*}}+e\right)=\vartheta(e)\,,\end{align*}
 which Fay shows is not the case, so $Jy$ cannot cancel. We know
$Jy\notin B$, so by Proposition \vref{pro:TopRedBox}, this also
leads to a contradiction, and so none of the zeroes and poles cancel.
Thus, $K^{a}(\cdot,\, y)$ has $n+1$ zeroes and poles.

We give a sketch proof that the poles are as stated. We use the alternate
characterisation of $K^{a}(x,\, y)$ given in \cite[Prop. 6.15]{Fay},
that is, \[
K^{a}(x,\, y)=\overline{\left(\frac{\Lambda_{a}(y,\, Jx)}{\Omega_{Ja-a}(y)}\right)}\,.\]
 Note that the notation here is partly that used in Fay, and partly
that used in this paper. In particular, $\Lambda$ and $\Omega$ are
as defined in Propositions 2.9 and 6.15 of Fay respectively (the definitions
are too complicated to replicate here). Clearly, $\Omega_{Ja-a}(y)$
has no dependence on $x$, so has no direct bearing on the poles in
$x$ of $K^{a}$. However, we note that the divisor $\mathcal{A}$
used in the construction of $\Lambda$ is the zero divisor of $\Omega_{Ja-a}$,
which is precisely the critical divisor of $g(\cdot,\, a)$. We then
use the description of $\textrm{div}\Lambda_{a}$ from \cite[Prop. 2.9]{Fay}
to see that for fixed $y$, the poles of $\Lambda_{a}\left(y,\, J(\cdot)\right)$
are precisely \[
\left\{ Jx\right\} \cup J\left(\mathcal{A}\right)=\left\{ Jx,\, P_{1}(a),\,\ldots,\, P_{n}(a)\right\} \,,\]
 where the $P_{i}(a)$s are as required. 
\end{proof}
We will write $P_{i}(b)=P_{i}$, for brevity.

\begin{thm}
\label{thm:TheEpsilon}Let $a_{1}^{0},\,\ldots,\, a_{2n}^{0}$ be
points in $R$ such that \[
P_{1},\,\ldots,P_{n},\, Jb,\, Ja_{1}^{0},\,\ldots,\, Ja_{2n}^{0}\]
 are all distinct. Let $\left\{ e_{1},\, e_{2}\right\} $ denote the
standard basis for $\mathbb{C}^{2}$ and let \[
\gamma_{1}^{0}=\cdots=\gamma_{n}^{0}=e_{1}\,,\quad\gamma_{n+1}^{0}=\cdots=\gamma_{2n}^{0}=e_{2}\,.\]
 There exists an $\epsilon>0$ so that if $\left|a_{j}^{0}-a_{j}\right|,\,\left\Vert \gamma_{j}^{0}-\gamma_{j}\right\Vert <\epsilon$,
and \begin{equation}
h(z)=\sum_{j=1}^{2n}c_{j}K^{b}(z,\, a_{j})\gamma_{j}+v\label{eq:poles}\end{equation}
 is a $\mathbb{C}^{2}$-valued meromorphic function which does not
have poles at $P_{1},\,\ldots,\, P_{n}$, then $h$ is constant; that
is, each $c_{j}=0$.

Further, if $h\neq0$ has a representation as in \eqref{eq:poles},
and there exists $z^{\prime}\in R\backslash\{b\}$ such that \[
h(z)K^{b}(z,\, z^{\prime})=\sum c_{j}^{\prime}K^{b}(z,\, a_{j})\gamma_{j}+v^{\prime}\]
 then $h$ is constant, $z^{\prime}=a_{j}$ for some $j$, $c_{j}^{\prime}\gamma_{j}=h$,
and all other terms are zero. 
\end{thm}
This theorem can be seen as a result about meromorphic functions on
$Y$, so we view $z$ as a local co-ordinate on $Y$. If we're only
interested in values of $z$ near one of $P_{1},\,\ldots,\, P_{n}$,
we can assume $z,\, P_{1},\,\ldots,P_{n},\, Ja_{1},\,\ldots,\, Ja_{2n}$
are in a single chart $U\subseteq J(R)$ ($U$ is open and simply
connected)

A useful tool in the proof of this theorem is the residue of $K^{b}$.
We know that so long as $a\notin\left\{ b,\, P_{1},\,\ldots,\, P_{n}\right\} $,
$K^{b}(\cdot,\, a)$ has only simple poles, so we know that in a small
enough neighbourhood of $P_{j}$, \[
(z-P_{j})K^{b}(z,\, a)\]
 is a holomorphic function in $z$. Let $R_{j}(a)$ denote the value
of this function at $P_{j}$.

We will need the following lemma.

\begin{lem}
\label{lem:Residue}The residue $R_{j}(a)$ varies continuously with
$a$. 
\end{lem}
\begin{proof}
Consider the theta function representation of $K^{b}(z,\, a)$. The
function \[
f(z)=\vartheta\left(\chi(z)+\chi(b)^{*}+e\right)\]
 is analytic and single valued on $U$, and vanishes with order one
at $P_{j}$, so can be written as \[
f(z)=(z-P_{j})f_{j}(z)\]
 for some $f_{j}$ analytic on $U$, and non-vanishing at $P_{j}$.
Given a set $W\subseteq U$, let $W^{*}=\left\{ \overline{z}:\, z\in W\right\} $.
Choose neighbourhoods $V_{j},\, W$ of $U$ so that $F:V_{j}\times W^{*}\to\mathbb{C}$
given by \begin{multline*}
F(z,\, a)=f(z)K^{b}(z,\, a)\\
=\frac{\vartheta\left(\chi(z)+\chi(a)^{*}+e\right)\vartheta\left(\chi(b)+\chi(b)^{*}+e\right)\vartheta_{*}\left(\chi(b)+\chi(a)^{*}\right)\vartheta_{*}\left(\chi(z)+\chi(b)^{*}\right)}{\vartheta\left(\chi(b)+\chi(a)^{*}+e\right)\vartheta_{*}\left(\chi(z)+\chi(a)^{*}\right)\vartheta_{*}\left(\chi(b)+\chi(b)^{*}\right)}\end{multline*}
 is analytic in $(z,\,\overline{a})$. Rewriting gives \[
(z-P_{j})K^{b}(z,\, a)=\frac{F(z,\, a)}{f_{j}(z)}\]
 The lemma follows from the fact that the right hand side is analytic
in $(z,\,\overline{a})$. 
\end{proof}
We can now prove Theorem \ref{thm:TheEpsilon}.

\begin{proof}
{[}Proof of Theorem \ref{thm:TheEpsilon}]We can assume $\epsilon$
is small enough that \[
P_{1},\,\ldots,\, P_{n},\, Ja_{1},\,\ldots,\, Ja_{2n}\]
 are distinct. We define \[
\mathfrak{R}_{1}=\left(\begin{array}{ccc}
R_{1}(a_{1}) & \cdots & R_{1}(a_{n})\\
\vdots & \ddots & \vdots\\
R_{n}(a_{1}) & \cdots & R_{n}(a_{n})\end{array}\right)\]
 and \[
\mathfrak{R}_{2}=\left(\begin{array}{ccc}
R_{1}(a_{n+1}) & \cdots & R_{1}(a_{2n})\\
\vdots & \ddots & \vdots\\
R_{n}(a_{n+1}) & \cdots & R_{n}(a_{2n})\end{array}\right)\,,\]
 where $R_{j}(a)$ is the residue of $K^{b}(\cdot,\, a)$ at $P_{j}$,
as before.

To see that $\mathfrak{R}_{1}$ is invertible, let \[
c=\left(\begin{array}{c}
c_{1}\\
\vdots\\
c_{n}\end{array}\right)\]
 and \[
f_{c}=\sum_{j=1}^{n}c_{j}K^{b}(\cdot,\, a_{j})\,.\]
 Note that $\mathfrak{R}_{1}c=0$ if and only if $f_{c}$ does not
have poles at any $P_{j}$. Now, if this is the case, then $f_{c}$
can only have poles at $Ja_{1},\,\ldots,\, Ja_{n}$, and simple poles
at that, but this is only $n$ points, so by Proposition \ref{pro:TopRedBox},
$f_{c}$ must be constant. We know that $K^{b}(\cdot,\, b)=1$, so
we can say that \[
0=c_{0}K^{b}(\cdot,\, b)+c_{1}K^{b}(\cdot,\, a_{1})+\cdots+c_{n}K^{b}(\cdot,\, a_{n})\,.\]
 However, we know that $K^{b}(\cdot,\, b),\, K^{b}(\cdot,\, a_{1}),\,\ldots,\, K^{b}(\cdot,\, a_{n})$
are linearly independent, so $c=0$. Therefore $\mathfrak{R}_{1}$
is invertible, and by a similar argument $\mathfrak{R}_{2}$ is invertible.

Now, consider the function $F$ defined for $\gamma_{j}$ near $\gamma_{j}^{0}$
by \[
F=\left(\begin{array}{cccc}
R_{1}(a_{1})\gamma_{1} & \cdots & \cdots & R_{1}(a_{2n})\gamma_{2n}\\
\vdots & \cdots & \cdots & \vdots\\
R_{2}(a_{1})\gamma_{1} & \cdots & \cdots & R_{2}(a_{2n})\gamma_{2n}\end{array}\right)\,.\]
 We define $F_{0}$ similarly, using $a_{j}^{0}$ and $\gamma_{j}^{0}$.
We can see that $F$ is an $n\times2n$ matrix with entries from $\mathbb{C}^{2}$,
so can be regarded as a $2n\times2n$ matrix. We know that $F$ varies
continuously with each $\gamma_{j}$, and by Lemma \ref{lem:Residue},
varies continuously with each $a_{j}$. Also, we see that, by regarding
$F_{0}$ as a $2n\times2n$ matrix, the rows of $F_{0}$ can be shuffled
to give \[
\left(\begin{array}{cc}
\mathfrak{R}_{1} & 0\\
0 & \mathfrak{R}_{2}\end{array}\right)\]
 which is invertible, so $F_{0}$ is invertible. We can therefore
choose $\epsilon>0$ small enough that if $\left|a_{j}-a_{j}^{0}\right|,\,\left\Vert \gamma_{j}-\gamma_{j}^{0}\right\Vert <\epsilon$
for all $j$, then $F$ is invertible.

If the $a_{j}$ and $\gamma_{j}$ are chosen such that $F$ is invertible
and \[
h(z)=\sum_{j=1}^{n}c_{j}K^{b}(z,\, a_{j})\gamma_{j}+v\]
 does not have poles at $P_{j}$, then \[
0=\left(\begin{array}{c}
\sum_{j=1}^{n}c_{j}R_{1}(a_{j})\gamma_{j}\\
\vdots\\
\sum_{j=1}^{n}c_{j}R_{n}(a_{j})\gamma_{j}\end{array}\right)=F\left(\begin{array}{c}
c_{1}\\
\vdots\\
c_{2n}\end{array}\right)=Fc\,,\]
 so $c=0$, and $h$ is constant.

Now we prove the second part of the theorem. Note that the proof of
this part only assumes that the result of the first part holds, not
the assumptions on $a_{j}$ and $\gamma_{j}$ used to prove it. Suppose
$h\neq0$ and there exists $z^{\prime}\in R\backslash\{b\}$ such
that \[
h(z)K^{b}(z,\, z^{\prime})=\sum c_{j}^{\prime}K^{b}(z,\, a_{j})\gamma_{j}+v^{\prime}\,.\]
 We can see that $P_{1},\,\ldots,\, P_{n}$ are not poles of $h$,
since by the assumptions on the distinctness of the $P_{k}$s and
$a_{j}$s, the right hand side has a pole of order at most one at
each $P_{k}$, whilst the left hand side has poles of order at least
one at each of these points. Therefore, since $h$ has a representation
as in the first part of the theorem, $h$ is constant. 
\end{proof}

\section{Representations}

This paper inherits much of its structure from \cite{DritschelRationalDilation},
and in particular, the results in this section are analogues of results
from that paper. In fact, in some cases, the proofs in \cite{DritschelRationalDilation}
do not use the connectivity of $X$, so can be used to prove their
analogues here simply by noting this fact. In these cases, the proofs
are omitted.

\subsection{Kernels, Realisations and Interpolation}

We note, for those who are interested, that many of these results
have a similar flavour to some of the Schur-Agler class results from
\cite{DritschelInterpolation}, although we shall not use any of these
results directly.

\begin{lem}
If $F\in M_{2}\left(\mathbb{H}(X)\right)$, then there exists a $\rho>0$
such that \[
I-\rho^{2}F(z)F(w)^{*}\in\mathcal{C}\,.\]
 
\end{lem}
\begin{thm}
\label{thm:1.1Baby}If there is a function $F:R\to M_{2}(\mathbb{C})$
which is analytic in a neighbourhood of $X$ and unitary valued on
$B$, such that $\rho_{F}<1$, then there exists an operator $T\in\mathcal{B}(H)$
for some Hilbert space $H$, such that the homomorphism $\pi:\mathcal{R}(X)\to\mathcal{B}(H)$
given by $\pi(p/q)=p\left(T\right)\cdot q\left(T\right)^{-1}$ is
contractive, but not completely contractive. 
\end{thm}
Later on in this section, we will need to work with matrix valued
Herglotz representations, so we will need some results about matrix-valued
measures. Given a compact Hausdorff space $X$, an $m\times m$ matrix-valued
measure \[
\mu=\left(\mu_{jl}\right)_{j,l=1}^{m}\]
 is an $m\times m$ matrix whose entries $\mu_{jl}$ are complex-valued
Borel measures on $X$. The measure $\mu$ is positive (we write $\mu\geq0$)
if for each function $f:X\to\mathbb{C}^{m}$ \[
f=\left(\begin{array}{c}
f_{1}\\
\vdots\\
f_{m}\end{array}\right)\,,\]
 we have \[
0\leq\sum_{j,l}\int_{X}\overline{f_{j}}f_{l}d\mu_{jl}\,.\]
 The positive measure $\mu$ is bounded by $M>0$ if \[
MI_{m}-\left(\mu_{jl}(X)\right)\geq0\]
 is positive semi-definite, where $I_{m}$ is the $m\times m$ identity
matrix.

\begin{lem}
The $m\times m$, matrix-valued measure $\mu$ is positive if and
only if for each Borel set $\omega$ the $m\times m$ matrix \[
\left(\mu_{jl}(\omega)\right)\]
 is positive semi-definite.

Further, if there is a $\kappa$ so that each diagonal entry $\mu_{jj}(X)\leq\kappa$,
then each entry $\mu_{jl}$ of $\mu$ has total variation at most
$\kappa$. Particularly, if $\mu$ is bounded by $M$, then each entry
has variation at most $M$. \end{lem}

\begin{lem}
If $\mu^{n}$ is a sequence of positive $m\times m$ matrix-valued
measures on $X$ which are all bounded above by $M$, then $\mu^{n}$
has a weak-$*$ convergent sub-sequence, that is, there exists a positive
$m\times m$ matrix-valued measure $\mu$, such that for each pair
of continuous functions $f,\, g:X\to\mathbb{C}^{m}$, \[
\sum_{j,\, l}\int_{X}f_{l}\overline{g_{j}}d\mu_{jl}^{n_{k}}\to\sum_{j,\, l}\int_{X}f_{l}\overline{g_{j}}d\mu_{jl}\,.\]\end{lem}

\begin{lem}
If $\mu$ is a positive $m\times m$ matrix-valued measure on $X$,
then the diagonal entries, $\mu_{jj}$ are positive measures. Further,
with $\nu=\sum_{j}\mu_{jj}$, there exists an $m\times m$ matrix-valued
function $\Delta:X\to M_{m}(\mathbb{C})$ so that $\Delta(x)$ is
positive semi-definite for each $x\in X$ and $d\mu=\Delta d\nu$
-- that is, for each pair of continuous functions, $f,\, g:X\to\mathbb{C}^{m}$,
\[
\sum_{j,\, l}\int_{X}\overline{g_{j}}f_{l}d\mu_{jl}=\sum_{j,\, l}\int_{X}\overline{g_{j}}\,\Delta_{jl}f_{l}d\nu\,.\]

\end{lem}
A key result of this section is the existence of a Herglotz representation
for well behaved inner functions, as follows.

\begin{prop}
\label{pro:HerglotzRep}Suppose $F$ is a $2\times2$ matrix-valued
function analytic in a neighbourhood of $R$, $F$ is unitary valued
on $B$, and $F(b)=0$. If $\rho_{F}=1$ and if $S\subseteq R$ is
a finite set, then there exists a probability measure $\mu$ on $\Pi$
and a positive kernel $\Gamma:S\times S\times\Pi\to\mathbb{C}$ so
that \[
1-F(z)F(w)^{*}=\int_{\Pi}\left(1-\psi_{p}(z)\overline{\psi_{p}(w)}\right)\Gamma(z,\, w;\, p)d\mu(p)\,.\]

\end{prop}
\begin{proof}
The proof of this result is almost identical to that of \cite[Prop. 5.6]{DritschelRationalDilation},
except that functions required to vanish at zero, are now required
to vanish at $b$ instead. 
\end{proof}
Another tool that will prove useful is transfer function representations.
For our purposes it will suffice to work with relatively simple colligations.
We will define a unitary colligation $\Sigma$ by $\Sigma=\left(U,\, K,\,\mu\right)$,
where $\mu$ is a probability measure on $\Pi$, $K$ is a Hilbert
space, and $U$ is a linear operator, defined by \[
U=\left(\begin{array}{cc}
\mathbf{A} & \mathbf{B}\\
\mathbf{C} & \mathbf{D}\end{array}\right)\in\mathcal{B}\left(\begin{array}{c}
L^{2}(\mu)\otimes K\\
\oplus\\
\mathbb{C}^{2}\end{array}\right)\,,\]
 where $L^{2}\otimes K$ can be regarded as $K$ valued $L^{2}$.

We define $\Phi:R\to\mathcal{B}\left(L^{2}(\mu)\otimes K\right)$
by \[
\left(\Phi(z)\, f\right)(p)=\psi_{p}(z)\, f(p)\,.\]
 From here, we define the transfer function associated to $\Sigma$
by \[
W_{\Sigma}(z)=\mathbf{D}+\mathbf{C}\Phi(x)\left(I-\Phi(z)\mathbf{A}\right)^{-1}\Phi(z)\mathbf{B}\,.\]
 We can see that as $\mathbf{A}$ is a contraction and $\Phi(z)$
is a strict contraction, the inverse in $W_{\Sigma}$ exists for any
$z\in R$.

\begin{prop}
\label{pro:TransferContractive}The transfer function is contraction
valued, that is, $\left\Vert W_{\Sigma}(z)\right\Vert \leq1$ for
all $z\in R$. In fact for all $z,\, w\in R$ \[
I-W_{\Sigma}(z)W_{\Sigma}(w)^{*}=\mathbf{C}\left(I-\Phi(z)\mathbf{A}\right)^{-1}\left[I-\Phi(z)\Phi(w)^{*}\right]\left(I-\Phi(w)\mathbf{A}\right)^{*-1}\mathbf{C}^{*}\,.\]

\end{prop}
Note that if we define $H(w)=\left(I-\mathbf{A}^{*}\Phi(w)^{*}\right)^{-1}\mathbf{C}^{*}$,
for $w$ fixed, $H(w)^{*}$ is a function on $\Pi$, so we write $H_{p}(w)^{*}$.
We can see that by considering $L^{2}(\mu)\otimes K$ as a measure
space, Proposition \vref{pro:TransferContractive} gives \[
I-W(z)W(w)^{*}=\int_{\Pi}\left(1-\psi_{p}(z)\overline{\psi_{p}(w)}\right)H_{p}(z)H_{p}(w)^{*}d\mu(p)\,.\]

\begin{prop}
\label{pro:TransferRepn}If $S\subseteq R$ is a finite set, $W:S\to M_{2}(\mathbb{C})$
and there is a positive kernel $\Gamma:S\times S\times\Pi\to M_{2}(\mathbb{C})$
such that \[
I-W(z)W(w)^{*}=\int_{\Pi}\left(1-\psi_{p}(z)\overline{\psi_{p}(w)}\right)\Gamma(z,\, w;\, p)\, d\mu(p)\]
 for all $z,\, w\in S$, then there exists $G:R\to M_{2}(\mathbb{C})$
such that $G$ is analytic, $\left\Vert G(z)\right\Vert \leq1$ and
$G(z)=W(z)$ for $z\in S$. Indeed, there exists a finite-dimensional
Hilbert space $K$ (dimension at most $2\left|S\right|$) and a unitary
colligation $\Sigma=\left(U,\, K,\,\mu\right)$ so that \[
G=W_{\Sigma}\,,\]
 and hence there exists $\Delta:R\times R\times\Pi\to M_{2}(\mathbb{C})$
a positive analytic kernel such that\[
I-G(z)G(w)^{*}=\int_{\Pi}\left(1-\psi_{p}(z)\overline{\psi_{p}(w)}\right)\Delta(z,\, w;\, p)\, d\mu(p)\]
 for all $z,\, w\in R$. 
\end{prop}
The proof is as in \cite{DritschelRationalDilation}, although for
our purposes it makes sense to use the version of Kolmogorov's theorem
in \cite[Thm. 2.62]{AglerPick}.

\subsection{Uniqueness}

\begin{prop}
\label{pro:Tight}Suppose $F:R\to M_{2}(\mathbb{C})$ is analytic
in a neighbourhood of $X$, unitary on $B$, and with a standard zero
set. Then there exists a set $S\subseteq R$ with $2n+3$ elements
such that, if $Z:R\to M_{2}(\mathbb{C})$ is contraction-valued, analytic,
and $Z(z)=F(z)$ for $z\in S$, then $Z=F$. 
\end{prop}
\begin{proof}
Let $K^{b}$ denote the Fay kernel for $R$ defined in Theorem \vref{thm:FayKernel}.
That is, $K^{b}$ is the reproducing kernel for the Hilbert space
\[
\mathbb{H}^{2}:=\mathbb{H}^{2}(R,\,\omega_{b})\]
 of functions analytic in $R$ with $L^{2}(\omega_{b})$ boundary
values. Let $\mathbb{H}_{2}^{2}$ denote $\mathbb{C}^{2}$-valued
$\mathbb{H}^{2}$. Since $F$ is unitary valued on $B$, the mapping
$V$ on $\mathbb{H}_{2}^{2}$ given by $VG(z)=F(z)G(z)$ is an isometry.
Also, as we will show, the kernel of $V^{*}$ is the span of \[
\mathfrak{V}:=\left\{ K^{b}(\cdot,\, a_{j})\gamma_{j}:\, j=1,\,\ldots,\,2n+2\right\} \,,\]
 where $F(a_{j})^{*}\gamma_{j}=0$ and $\gamma_{j}\neq0$; that is,
$(a_{j},\,\gamma_{j})$ is a zero of $F^{*}$.

We note, for future use, that if $\varphi$ is a scalar-valued analytic
function on a neighbourhood of $R$, with no zeroes on $B$, and zeroes
$w_{1},\,\ldots,\, w_{n}\in R$, all of multiplicity one, and $f\in\mathbb{H}^{2}$
has roots at all these $w_{i}$s, then $f=\varphi g$ for some $g\in\mathbb{H}^{2}$.

Now, suppose $\psi\in\mathbb{H}^{2}$ and for all $h\in\mathbb{H}^{2}$
we have $\left\langle \psi,\,\varphi h\right\rangle =0$. Since the
set \[
\mathfrak{K}:=\left\{ K^{b}(\cdot,\, w_{j}):\,1\leq j\leq n\right\} \]
 is linearly independent, we know there is some linear combination
\[
f=\psi-\sum_{j=1}^{n}c_{j}K^{b}(\cdot,\, w_{j})\,,\]
 so that $f(w_{j})=0$ for all $j$, and so $f=\varphi g$ for some
$g$. Since \[
\left\langle K^{b}(\cdot,\, w_{j}),\,\varphi h\right\rangle =\overline{\varphi(w_{j})\, h(w_{j})}=0\]
 for each $j$ and $h$, it follows that $\left\langle f,\,\varphi h\right\rangle =0$
for all $h$. In particular, if $h=g$ (the $g$ we found earlier),
then \[
\left\langle \varphi g,\,\varphi g\right\rangle =\left\langle f,\,\varphi g\right\rangle =0\,,\]
 so $g\equiv0$, and so \begin{equation}
0=f=\psi-\sum_{j=1}^{n}c_{j}K^{b}(\cdot,\, w_{j})\,.\label{eq:RootSpan}\end{equation}
 This tells us that $\psi$ is in the span of $\mathfrak{K}$, so
$\mathfrak{K}$ is a basis for the orthogonal complement of $\left\{ \varphi h:\, h\in\mathbb{H}^{2}\right\} $.

We now find the kernel of $V^{*}$. Write $a_{2n+1}=a_{2n+2}=b$.
Since $F(b)=0$, there is a function $H$ analytic in a neighbourhood
of $X$ so that $F(z)=(z-b)H(z)$. The function $\varphi(z)=(z-b)\det\left(H(z)\right)$
satisfies the hypothesis of the preceding paragraph.

Let \[
G:=\left(\begin{array}{cc}
h_{22} & -h_{12}\\
-h_{21} & h_{11}\end{array}\right)\,,\]
 where $H=\left(h_{jl}\right)$. Then \[
FG=(z-b)HG=(z-b)\det(H)I\,,\]
 where $I$ is the $2\times2$ identity matrix.

Now, suppose $x\in\mathbb{H}_{2}^{2}$ and $V^{*}x=0$. Let $x_{1},\, x_{2}$
be the co-ordinates of $x$. For each $g\in\mathbb{H}_{2}^{2}$, \begin{align*}
0 & =\left\langle Gg,\, V^{*}x\right\rangle \\
 & =\left\langle VGg,\, x\right\rangle \\
 & =\left\langle (z-b)\det(H)g,\, x\right\rangle \\
 & =\left\langle (z-b)\det(H)g_{1},\, x_{1}\right\rangle +\left\langle (z-b)\det(H)g_{2},\, x_{2}\right\rangle \,.\end{align*}
 It therefore follows from the discussion leading up to \eqref{eq:RootSpan}
that both $x_{1}$ and $x_{2}$ are in the span of \[
\left\{ K^{b}(\cdot,\, a_{j}):\,1\leq j\leq2n+2\right\} \,,\]
 so \[
x\in\textrm{Span}\left\{ K^{b}(\cdot,\, a_{j})v:\,1\leq j\leq2n+2,\, v\in\mathbb{C}^{2}\right\} \,.\]
 In particular, there exist vectors $v_{j}\in\mathbb{C}^{2}$ such
that \[
x=\sum_{j=1}^{2n+2}K^{b}(\cdot,\, a_{j})\, v_{j}\,.\]

We can check that $V^{*}vK^{b}(\cdot,\, a)=F(a)^{*}vK^{b}(\cdot,\, a)$,
and $F(b)^{*}=0$, so \[
0=V^{*}x=\sum_{j=1}^{2n}F(a_{j})^{*}v_{j}K^{b}(\cdot,\, a_{j})\,,\]
 but the $K^{b}(\cdot,\, a_{j})$s are linearly independent, so $F(a_{j})^{*}v_{j}=0$
for all $j$. Conversely, if $F(a_{j})^{*}v_{j}=0$ then $V^{*}v_{j}K^{b}(\cdot,\, a_{j})=0$,
so the kernel of $V^{*}$ is spanned by $\mathfrak{V}$.

Now, since $V$ is an isometry, $I-VV^{*}$ is the projection onto
the kernel of $V^{*}$, which by the above argument has dimension
$2n+2$, so $I-VV^{*}$ has rank $2n+2$. So, for any finite set $A\subseteq R$,
the block matrix with $2\times2$ entries \begin{align*}
M_{A}= & \left(\left[\left\langle \left(I-VV^{*}\right)K^{b}(\cdot,\, w)e_{j},\, K^{b}(\cdot,\, z)e_{l}\right\rangle \right]_{j,\, l=1,\,2}\right)_{z,\, w\in A}\\
= & \left(\left(I-F(z)F(w)^{*}\right)K^{b}(z,\, w)\right)_{z,\, w\in A}\end{align*}
 has rank at most $2n+2$. In particular, if $A=\left\{ a_{1},\,\ldots,\, a_{2n+2}\right\} \,,$
then $M_{A}$ has rank exactly $2n+2$. Choose $a_{2n+3},\, a_{2n+4}$
distinct from $a_{1},\,\ldots,\, a_{2n+2}$ so that \[
S=\left\{ a_{1},\,\ldots,\, a_{2n+2},\, a_{2n+3},\, a_{2n+4}\right\} \]
 has $2n+3$ \emph{distinct} points. Since $A\subseteq S$, $M_{S}$
has rank at least $2n+2$. However, by the above discussion, its rank
cannot exceed $2n+2$, so its rank must be exactly $2n+2$.

The matrix $M_{S}$ is $\left(4n+6\right)\times\left(4n+6\right)$,
(a $\left(2n+3\right)\times\left(2n+3\right)$ matrix with $2\times2$
matrices as its entries), and $M_{S}$ has rank $2n+2$, so must have
nullity (that is, kernel dimension) $2n+4$. Further, the subspace
\[
\mathcal{L}_{1}:=\left\{ \underbrace{\left(\begin{array}{c}
\left(\begin{array}{c}
\alpha_{1}\\
0\end{array}\right)\\
\vdots\\
\left(\begin{array}{c}
\alpha_{2n+3}\\
0\end{array}\right)\end{array}\right)}_{=\alpha\otimes e_{1}}:\,\alpha=\left(\begin{array}{c}
\alpha_{1}\\
\vdots\\
\alpha_{2n+3}\end{array}\right)\in\mathbb{C}^{2n+3}\right\} \]
 is $2n+3$ dimensional, so there exists a non-zero $x_{1}=y_{1}\otimes e_{1}$
in $\mathcal{L}_{1}$ which is in the kernel of $M_{S}$. Similarly,
$\mathcal{L}_{2}:=\left\{ \alpha\otimes e_{2}:\,\alpha\in\mathbb{C}^{2n+3}\right\} $
contains some $x_{2}$ in the kernel of $M_{S}$.

Let $x=(x_{1}\quad x_{2})$, so $x$ is the $(4n+6)\times2$ matrix
\[
x=\left(\begin{array}{c}
\left(\begin{array}{cc}
(y_{1})_{1} & 0\\
0 & (y_{2})_{1}\end{array}\right)\\
\left(\begin{array}{cc}
(y_{1})_{2} & 0\\
0 & (y_{2})_{2}\end{array}\right)\\
\vdots\\
\left(\begin{array}{cc}
(y_{1})_{2n+3} & 0\\
0 & (y_{2})_{2n+3}\end{array}\right)\end{array}\right)\,.\]
 It will be more convenient to refer to $2\times2$ blocks in $x$
by their corresponding point in $S$, rather than their number, so
we say \[
x(w)=\left(x_{1}(w)\quad x_{2}(w)\right)=\left(\begin{array}{cc}
y_{1}(w) & 0\\
0 & y_{2}(w)\end{array}\right)\,.\]
 In this notation, the identity $M_{S}x=0$ becomes \[
\sum_{w\in S}K^{b}(z,\, w)\, x(w)=F(z)\sum_{w\in S}K^{b}(z,\, w)\, F(w)^{*}\, x(w)\]
 for each $z$.

Now, suppose $Z:R\to M_{2}(\mathbb{C})$ is analytic, contraction
valued, and $Z(z)=F(z)$ for $z\in S$. The operator $W$ of multiplication
by $Z$ on $\mathbb{H}_{2}^{2}$ is a contraction and \[
W^{*}K^{b}(\cdot,\, w)v=Z(w)^{*}vK^{b}(\cdot,\, w)\,.\]
 Given $\zeta\in R$, $\zeta\notin S$, let $S^{\prime}=S\cup\{\zeta\}$
and consider the decomposition of \[
N_{\zeta}=\left(\left(I-Z(z)Z(w)^{*}\right)K^{b}(z,\, w)\right)_{z,\, w\in S^{\prime}}\]
 into blocks labelled by $S$ and $\{\zeta\}$. Thus $N_{\zeta}$
is a $(2n+4)\times(2n+4)$ matrix with $2\times2$ block entries.
The upper left $(2n+3)\times(2n+3)$ block is simply $M_{S}$, as
$Z(z)=F(z)$ for $z\in S$.

Let \[
x^{\prime}=\left(\begin{array}{c}
x\\
\left(\begin{array}{cc}
0 & 0\\
0 & 0\end{array}\right)\end{array}\right)\,.\]
 Since $N_{\zeta}$ is positive semi-definite and $M_{S}x=0$, it
can be shown that $N_{\zeta}x^{\prime}=0$. An examination of the
last two entries of the equation $N_{\zeta}x^{\prime}=0$ gives \begin{equation}
\sum_{w\in S}K^{b}(\zeta,\, w)x(w)=Z(\zeta)\sum_{w\in S}Z(w)^{*}K^{b}(\zeta,\, w)x(w)\,.\label{eq:ZConfined}\end{equation}

The left hand side of \eqref{eq:ZConfined} is a rank 2, $2\times2$
matrix at all but countably many $\zeta$, as it is a diagonal matrix
whose diagonal elements are of the form \[
\sum_{w\in S}K^{b}(\zeta,\, w)y_{i}(w)\,;\]
 that is, linear combinations of $K^{b}(\zeta,\, w)$s. If such a
function is zero at an uncountable number of $\zeta$s, it is identically
zero, which is impossible, as the $K^{b}(\cdot,\, w)$s are linearly
independent and the $y_{i}(w)$s are not all zero. We can now see
that \[
\sum_{w\in S}Z(w)^{*}K^{b}(\zeta,\, w)x(w)\]
 is invertible at all but countably many $\zeta$, so \begin{align*}
Z(\zeta)= & \sum_{w\in S}K^{b}(\zeta,\, w)x(w)\left(\sum_{w\in S}Z(w)^{*}K^{b}(\zeta,\, w)x(w)\right)^{-1}\\
= & \sum_{w\in S}K^{b}(\zeta,\, w)x(w)\left(\sum_{w\in S}F(w)^{*}K^{b}(\zeta,\, w)x(w)\right)^{-1}\\
= & F(\zeta)\end{align*}
 at all but finitely many $\zeta$, so $Z=F$. 
\end{proof}
We combine some of the preceding results to get the following.

\begin{thm}
\label{thm:CombinedTightRepn}Suppose $F$ is a $2\times2$ matrix-valued
function analytic in a neighbourhood of $R$, which is unitary-valued
on $B$, and with a standard zero set. If $\rho_{F}=1$, then there
exists a unitary colligation $\Sigma=(U,\, K,\,\mu)$ such that $F=W_{\Sigma}$,
and so that the dimension of $K$ is at most $4n+6$. In particular,
$\mu$ is a probability measure on $\Pi$ and there is an analytic
function $H:R\to L^{2}(\mu)\otimes M_{4n+6,\,2}(\mathbb{C})$, denoted
by $H_{p}(z)$, so that \[
I-F(z)F(w)^{*}=\int_{\Pi}\left(1-\psi_{p}(z)\overline{\psi_{p}(w)}\right)H_{p}(z)H_{p}(w)^{*}d\mu(p)\]
 for all $z,\, w\in R$. 
\end{thm}
\begin{proof}
Using Proposition \vref{pro:Tight}, choose a finite set $S\subseteq R$
such that if $G:R\to M_{2}(\mathbb{C})$ is analytic and contraction
valued, and $G(z)=F(z)$ for $z\in S$, then $G=F$. Using Proposition
\vref{pro:HerglotzRep}, we have a probability measure $\mu$ and
a positive kernel $\Gamma:S\times S\times\Pi\to M_{2}(\mathbb{C})$
such that \[
I-F(z)F(w)^{*}=\int_{\Pi}\left(1-\psi_{p}(z)\overline{\psi_{p}(w)}\right)\Gamma(z,\, w;\, p)\, d\mu(p)\]
 for all $z,\, w\in S$.

By Proposition \vref{pro:TransferRepn}, there exists a unitary colligation
$\Sigma=(U,\, K,\,\mu)$ so that $K$ is at most $4n+6$ dimensional,
and $W_{\Sigma}(z)=F(z)$ for $z\in S$. However, our choice of $S$
gives $W_{\Sigma}=F$ everywhere. We know $\Gamma(z,\, w;\, p)=H_{p}(z)H_{p}(w)^{*}$
for some $H_{p}$ by \cite[Thm. 2.62]{AglerPick}. 
\end{proof}
\begin{thm}
\label{thm:KSpan}Suppose $F$ is a $2\times2$ matrix-valued function
analytic in a neighbourhood of $R$, which is unitary valued on $B$,
with a standard zero set, and $\rho_{F}=1$, and is represented as
in Theorem \vref{thm:CombinedTightRepn}. Let $a_{2n+1}=a_{2n+2}=b$,
$\gamma_{2n+1}=e_{1}$, and $\gamma_{2n+2}=e_{2}$. Then there exists
a set $E$ of $\mu$ measure zero, such that for $p\notin E$, for
each $v\in\mathbb{C}^{4n+6}$, and for $l=0,\,1,\,\ldots,\, n$, the
vector function $H_{p}(\cdot)vK^{b}(\cdot,\, z_{l})$ is in the span
of $\left\{ K^{b}(\cdot,\, a_{j})\gamma_{j}\right\} $, where $z_{0}(p)(=b),\, z_{1}(p),\,\ldots,\, z_{n}(p)$
are the zeroes of $\psi_{p}$. Consequently, $H_{p}$ is analytic
on $R$ and extends to a meromorphic function on $Y$. 
\end{thm}
\begin{proof}
We showed in Proposition \vref{pro:Tight} that given a finite $Q\subseteq R$,
\[
M_{Q}=\left(\left(I-F(z)F(w)^{*}\right)K^{b}(z,\, w)\right)_{z,\, w\in Q}\]
 has rank at most $2n+2$, and that the range of $M_{Q}$ lies in
\begin{equation}
\mathfrak{M}:=\text{span}\left\{ \left(K^{b}(z,\, a_{i})\gamma_{i}\right)_{z\in Q}:\, i=1,\,\ldots,\,2n+2\right\} \,,\label{eq:BigM}\end{equation}
 thinking of $\left(K^{b}(z,\, a_{i})\gamma_{i}\right)_{z\in Q}$
as a column vector indexed by $Q$.

We then apply Theorem \vref{thm:CombinedTightRepn} to give \[
M_{Q}=\left(\int_{\Pi}H_{p}(z)\left(1-\psi_{p}(z)\overline{\psi_{p}(w)}\right)K^{b}(z,\, w)H_{p}(w)^{*}\, d\mu(p)\right)_{z,\, w\in Q}\,.\]

For each $p$, we define an operator $M_{p}\in\mathcal{B}(\mathbb{H}^{2})$
by \[
\left(M_{p}\, f\right)(x)=\psi_{p}(x)\, f(x)\,.\]
 Multiplication by $\psi_{p}$ is isometric on $\mathbb{H}^{2}$,
so $1-M_{p}M_{p}^{*}\geq0$, and so $\left(1-M_{p}M_{p}^{*}\right)\otimes E\geq0$,
where $E$ is the $m\times m$ matrix with all entries equal to 1.
From the reproducing property of $K^{b}$, we see that $M_{p}^{*}K^{b}(\cdot,\, z)=\overline{\psi_{p}(z)}K^{b}(\cdot,\, z)$.
Thus, if $Q$ is a set of $m$ points in $R$, and $c$ is the vector
$\left(K^{0}(\cdot,\, w)\right)_{w\in Q}$, then the matrix \[
P_{Q}(p)=\left\langle \left[(I-M_{p}M_{p}^{*})\otimes E\right]c,\, c\right\rangle =\left(\left[1-\psi_{p}(z)\overline{\psi_{p}(w)}\right]K^{b}(z,\, w)\right)_{z,\, w\in Q}\geq0\,.\]

If we set $\widetilde{Q}=Q\cup\{z_{j}\}$ for any $j=0,\,1,\,\ldots,\, n$,
then $P_{\widetilde{Q}}(p)\geq0$. Further, the upper $m\times m$
block equals $P_{Q}(p)$ and the right $m\times1$ column is $\left(K^{b}(z,\, z_{j}(p))\right)_{z\in Q}$.
Hence, as a vector, \[
\left(K^{b}(z,\, z_{j}(p))\right)_{z\in Q}\in\text{ran}P_{Q}(p)^{1/2}=\text{ran}P_{Q}(p)\,,\]
 for $j=0,\,1,\,\ldots,\, n$.

Since $P_{Q}\geq0$, \[
N_{Q}(p):=\left(H_{p}(z)\left(1-\psi_{p}(z)\overline{\psi_{p}(w)}\right)K^{b}(z,\, w)H_{p}(w)^{*}\right)_{z,\, w\in Q}\]
 is also positive semi-definite for each $p$. If $M_{Q}x=0$, then
\[
0=\int_{\Pi}\left\langle N_{Q}(p)x,\, x\right\rangle \, d\mu(p)\,,\]
 so that $\left\langle N_{Q}(p)x,\, x\right\rangle =0$ for almost
all $p$. It follows that $N_{Q}(p)x=0$ almost everywhere. Choosing
a basis for the kernel of $M_{Q}$, there is a set $E_{Q}$ of $\mu$
measure zero so that for $p\notin E_{Q}$, the kernel of $M_{Q}$
is a subspace of the kernel of $N_{Q}(p)$. For such $p$, the range
of $N_{Q}(p)$ is a subspace of the range of $M_{Q}$, so the rank
of $N_{Q}(p)$ is at most $2n+2$.

Further, if we let $D_{Q}(p)$ denote the diagonal matrix with ($2\times(4n+6)$
block) entries given by \[
D_{Q}(p)_{z,\, w}=\begin{cases}
H_{p}(z) & z=w\\
0 & z\neq w\end{cases}\,.\]
 then $N_{Q}(p)=D_{Q}(p)\, P_{Q}(p)\, D_{Q}(p)^{*}$. Since $P_{Q}(p)$
is positive semi-definite, we conclude that the range of $D_{Q}(p)\, P_{Q}(p)$
is in the range of $M_{Q}$. Therefore, since $\left(K^{b}(z,\, z_{j}(p))\right)_{z\in Q}$
is in the range of $P_{Q}(p)$, $\left(H_{p}(z)\, v\, K^{b}(z,\, z_{j}(p))\right)_{z\in Q}$
is in the range of $M_{Q}$ for every $v\in\mathbb{C}^{4n+6}$, and
$j=0,\,1,\,\ldots,\, n$.

Now suppose $Q_{m}\subseteq R$ is a finite set with \[
Q_{m}\subseteq Q_{m+1}\,,\quad Q_{0}=\left\{ a_{1},\,\ldots,\, a_{2n},a_{2n+1}(=b)\right\} \,,\]
 and \[
\mathcal{D}=\bigcup_{m\in\mathbb{N}}Q_{m}\]
 a determining set; that is, an analytic function is uniquely determined
by its values on $\mathcal{D}$. Since \[
\left(H_{p}(z)\, v\, K^{b}(z,\, z_{j}(p))\right)_{z\in Q_{m}}\in\text{ran}M_{Q_{m}}\subseteq\mathfrak{M}\,,\]
 we see that there are constants $c_{i}^{m}(p)$ such that \begin{equation}
H_{p}(z)\, v\, K^{b}(z,\, z_{j}(p))=\sum_{i=1}^{2n+2}c_{i}^{m}(p)\, K^{b}(z,\, a_{i})\,\gamma_{i}\,,\quad z\in Q_{m}\,.\label{eq:HpInKb}\end{equation}
 By linear independence of the $K^{b}(\cdot,\, a_{i})$s, the $c_{i}^{m}(p)$s
are uniquely determined when $n=0,\,1,\,\ldots$ by this formula.
Since $Q_{m+1}\supseteq Q_{m}$, we see that $c_{i}^{m+1}(p)=c_{i}^{m}(p)$
for all $m$, so there are unique constants $c_{i}(p)$ such that
\[
H_{p}(z)\, v\, K^{b}(z,\, z_{j}(p))=\sum_{i=1}^{2n+2}c_{i}(p)\, K^{b}(z,\, a_{i})\,\gamma_{i}\,,\quad z\in\mathcal{D}\,.\]
 Now, by considering this equation when $j=0$, and using the fact
that $K^{b}(\cdot,\, b)\equiv1$, we see that $H_{p}$ agrees with
an analytic function on a determining set. We can therefore assume
that $H_{p}$ is analytic for each $p\notin E$, and that \eqref{eq:HpInKb}
holds throughout $R$. Also, since the $K^{b}(\cdot,\, a_{i})$s extend
to meromorphic functions on $Y$, so must $H_{p}$. 
\end{proof}

\subsection{Diagonalisation}

\begin{lem}
\label{lem:UnitaryDiagTrick}Suppose $F$ is a matrix-valued function
on $R$ whose determinant is not identically zero. If there exists
a $2\times2$ unitary matrix $U$ and scalar valued functions $\phi_{1},\,\phi_{2}:R\to\mathbb{C}$
such that $F(z)F(w)^{*}=UD(z)D(w)^{*}U^{*}$, where \[
D:=\left(\begin{array}{cc}
\phi_{1} & 0\\
0 & \phi_{2}\end{array}\right)\,,\]
 then there exists a unitary matrix $V$ such that $F=UDV$. 
\end{lem}
\begin{proof}
The proof is as in \cite{DritschelRationalDilation}. We let $V=D(z)^{-1}U^{*}F(z)$,
which turns out to be constant and unitary. 
\end{proof}
\begin{thm}
\label{thm:rho=00003D00003D1=00003D00003D>diag}Suppose $F$ is a
$2\times2$ matrix-valued function which is analytic in a neighbourhood
of $R$, unitary valued on $B$, and has a standard zero set $\left(a_{j},\,\gamma_{j}\right)$,
$j=1,\,\ldots,\,2n$. Assume further that the $\left(a_{j},\,\gamma_{j}\right)$
have the property that if $h$ satisfies \[
h=\sum_{j=1}^{2n}c_{j}K^{b}(\cdot,\, a_{j})\gamma_{j}+v\,,\]
 for some $c_{1},\,\ldots,\, c_{2n}\in\mathbb{C}$ and $v\in\mathbb{C}^{2}$,
and $h$ does not have a pole at $P_{1},\,\ldots,\, P_{n}$, then
$h$ is constant.

Under these conditions, if $\rho_{F}=1$, then $F$ is diagonalisable,
that is, there exists unitary $2\times2$ matrices $U$, and $V$
and analytic functions $\phi_{1},\,\phi_{2}:R\to\mathbb{C}$ such
that \[
F=U\left(\begin{array}{cc}
\phi_{1} & 0\\
0 & \phi_{2}\end{array}\right)V=UDV\,.\]

\end{thm}
\begin{proof}
By Theorem \vref{thm:KSpan}, we may assume that except on a set $E$
of measure zero, if $h$ is a column of some $H_{p}$, then $h(\cdot)K^{b}(\cdot,\, z_{l}(s))\in\mathfrak{M}$
for $l=0,\,1,\,\ldots,\, n$.%
\footnote{Here $z_{0}(p)=b,\, z_{1}(p),\,\ldots,\, z_{n}(p)$ are the zeroes
of $\psi_{p}$, and $\mathfrak{M}$ is as defined in \eqref{eq:BigM}\vpageref{eq:BigM}.%
}

By hypothesis, $h$ (and so $H_{p}$) is constant. From Remark \vref{rem:DistinctZeros},
we can assume at least one of the zeroes of $\psi_{p}$ (say $z_{1}(p)$)
is not $b$. Thus, using the proof of the second part of Theorem \vref{thm:TheEpsilon},
we can show that if $h$ is not zero, then $z_{1}(p)=a_{j_{1}(p)}$
for some $j_{1}(p)$, and $h$ is a multiple of $\gamma_{j_{1}(p)}$.
Thus, every column of $H_{p}$ is a multiple of $\gamma_{j_{1}(p)}$.

Theorem \vref{thm:CombinedTightRepn} gives us\[
I-F(z)F(w)^{*}=\int_{\Pi}\left(1-\psi_{p}(z)\overline{\psi_{p}(w)}\right)H_{p}H_{p}^{*}d\mu(p)\,,\]
 and substituting $w=b$ gives\[
I=\int_{\Pi}H_{p}H_{p}^{*}d\mu(p)\]
 so \begin{equation}
F(z)F(w)^{*}=\int_{\Pi}\psi_{p}(z)\overline{\psi_{p}(w)}H_{p}H_{p}^{*}d\mu(p)\,.\label{eq:FinHs}\end{equation}

Since the columns of $H_{p}$ are all multiples of $\gamma_{j_{1}(p)}$,
$H_{p}H_{p}^{*}$ is rank one, and so can be written as $G(p)G(p)^{*}$
for a single vector $G(p)\in\mathbb{C}^{2}$. Consequently, \begin{equation}
F(z)F(w)^{*}=\int_{\Pi}\psi_{p}(z)\overline{\psi_{p}(w)}G(p)G(p)^{*}d\mu(p)\,.\label{eq:FinGs}\end{equation}

Since $F(a_{j})^{*}\gamma_{j}=0$ for all $j$, \eqref{eq:FinGs}
gives \[
0=\gamma_{j}^{*}F(a_{j})F(a_{j})^{*}\gamma_{j}=\int_{\Pi}\left|\psi_{p}(a_{j})\right|^{2}\left\Vert G(p)\gamma_{j}^{*}\right\Vert ^{2}d\mu(p)\,,\]
 so for each $j$, $\overline{\psi_{p}(a_{j})}G(p)^{*}\gamma_{j}=0$
for almost every $p$. So, apart from a set $Z_{0}\subseteq\Pi$ of
measure zero, $\overline{\psi_{p}(a_{j})}G(p)^{*}\gamma_{j}=0$ for
all $p$ and all $j$. Thus, by defining $G(p)=0$ for $p\in Z_{0}$,
we can assume that \eqref{eq:FinGs} holds and \[
\overline{\psi_{p}(a_{j})}G(p)^{*}\gamma_{j}=0\]
 for all values of $p$ and $j$.

Let $\Pi_{0}:=\left\{ p\in\Pi:\, G(p)=0\right\} $. If $p\notin\Pi_{0}$,
then for each $j$, either $\psi_{p}(a_{j})=0$ or $G(p)^{*}\gamma_{j}=0$.
Remember that $G_{p}$ is a multiple of $\gamma_{j_{1}(p)}$, and
no set of $n+1$ of the $\gamma_{j}$ all lie on the same line through
the origin. It follows that $\psi_{p}$ has zeroes at $b$, and $n$
of the $a_{j}$s (say $a_{j_{1}(p)},\,\ldots,\, a_{j_{n}(p)}$) and
$G(p)^{*}\gamma_{j}=0$ at $n$ of the $\gamma_{j}$s (say $\gamma_{j_{n+1}(p)},\,\ldots,\,\gamma_{j_{2n}(p)}$),
so these $\gamma_{j}$s must be orthogonal to $\gamma_{j_{1}(p)}$,
and so all lie on the same line through the origin. This tells us
that the zeroes of $\psi_{p}$ are precisely $b,\, a_{j_{1}(p)},\,\ldots,\, a_{j_{n}(p)}$,
so $z_{i}=a_{j_{i}(p)}$ for all $i$. We can also see that $\gamma_{j_{1}(p)},\,\ldots,\,\gamma_{j_{n}(p)}$
all lie on the same line through the origin, and so are orthogonal
to $\gamma_{j_{n+1}(p)},\,\ldots,\,\gamma_{j_{2n}(p)}$.

Let $\mathfrak{J}_{1}=\left\{ a_{j_{1}(p)},\,\ldots,\, a_{j_{n}(p)}\right\} $,
$\mathfrak{J}_{2}=\left\{ a_{j_{n+1}(p)},\,\ldots,\, a_{j_{2n}(p)}\right\} $,
let $\mathfrak{A}_{1}$ denote the one-dimensional subspace of $\mathbb{C}^{2}$
spanned by $\gamma_{j_{1}(p)}$ and $\mathfrak{A}_{2}$ denote the
one-dimensional space spanned by $\gamma_{j_{n+1}(p)}$.

If $q\notin\Pi_{0}$, then by arguing as above, either $G(q)\in\mathfrak{A}_{1}$
or $G(q)\in\mathfrak{A}_{2}$, and the zeroes of $\psi_{q}$ are in
$\mathfrak{J}_{2}$ or $\mathfrak{J}_{1}$ respectively. Hence, for
each $p$, one of the following must hold:
\begin{itemize}
\item (0): $G(p)=0$;
\item (1): $G(p)\in\mathfrak{A}_{1}$ and the zeroes of $\psi_{q}$ are
in $\mathfrak{J}_{2}\cup\{b\}$;
\item (2): $G(p)\in\mathfrak{A}_{2}$ and the zeroes of $\psi_{q}$ are
in $\mathfrak{J}_{1}\cup\{b\}$.
\end{itemize}
Define \begin{align*}
\Pi_{0}= & \left\{ p\in\Pi:\,\text{(0) holds}\right\} ,\\
\Pi_{1}= & \left\{ p\in\Pi:\,\text{(1) holds}\right\} ,\\
\Pi_{2}= & \left\{ p\in\Pi:\,\text{(2) holds}\right\} .\end{align*}
 If $p,\, q\in\Pi_{1}$ then $\psi_{p}$ and $\psi_{q}$ are equal,
up to multiplication by a unimodular constant, so we choose a $p^{1}\in\Pi_{1}$
and define $\psi_{1}=\psi_{p^{1}}$, so $\psi_{p}\overline{\psi_{p}}=\psi_{1}\overline{\psi_{1}}$
for all $p\in\Pi_{1}$. If $\Pi_{2}$ is non-empty, we do the same,
if not we define $\psi_{2}\equiv0$. We substitute this into \eqref{eq:FinHs}
to get \[
F(z)F(w)^{*}=h_{1}\psi_{1}(z)\overline{\psi_{1}(w)}h_{1}^{*}+h_{2}\psi_{2}(z)\overline{\psi_{2}(w)}h_{2}^{*}\,,\]
 where $h_{j}\in\mathfrak{A}_{j}$. Letting $z=w\in B$, we see that
$h_{1},\, h_{2}$ is an orthonormal basis for $\mathbb{C}^{2}$ (and
that $\psi_{2}\not\equiv0$), so we can apply Lemma \vref{lem:UnitaryDiagTrick},
and the result follows. 
\end{proof}

\section{The counterexample}

We now have all the tools we need to prove \prettyref{thm:ElTheoremGrande},
as introduced at the beginning of the paper. First, we constructed
$\Psi_{S,\mathbf{p}}$ in Lemma \vref{lem:PsiSp}, which is always
a $2\times2$ matrix-valued inner function. We then showed, in Lemma
\vref{lem:PsiSmp}, that there was a sequence $\Psi_{S_{m},\mathbf{p}}$,
such that each term had a standard zero set, with $S_{m}\neq S_{0}$
for all $m$, and such that both $S_{m}\to S_{0}$ and $\Psi_{S_{m},\mathbf{p}}\to\Psi_{S_{0},\mathbf{p}}$
as $m\to\infty$. We showed in Theorem \vref{thm:TheEpsilon}, that
if the zeroes $\left(a_{j},\,\gamma_{j}\right)$ of $\Psi_{S_{m},\mathbf{p}}$
are close enough to the zeroes of $\Psi_{S_{0},\mathbf{p}}$ (they
would be, for $m$ large enough, say $m=\mathbf{M}$) then any $\mathbb{C}^{2}$-valued
meromorphic function of the form \[
h(z)=\sum_{j=1}^{2n}c_{j}K^{b}(z,\, a_{j})\gamma_{j}+v\]
 with no poles at $P_{1},\,\ldots,\, P_{n}$ must be constant. Thus,
we take $\Psi=\Psi_{S_{\mathbf{M}},\mathbf{p}}$. Theorem \vref{thm:rho=00003D00003D1=00003D00003D>diag}
then tells us that if $\rho_{\Psi}=1$, then $\Psi$ is diagonalisable.
So if $\Psi$ is not diagonalisable, then $\rho_{\Psi}<1$. If $\rho_{\Psi}<1$,
Theorem \vref{thm:1.1Baby} tells us that there is an operator $T\in\mathcal{B}(H)$
for some $H$, such that the homomorphism $\pi:\mathcal{R}(X)\to\mathcal{B}(H)$
with $\pi(p/q)=p(T)\cdot q(T)^{-1}$ is contractive but not completely
contractive. Therefore, all that remains to be shown is that $\Psi$
is not diagonalisable.

\begin{thm}
$\Psi$ is not diagonalisable. 
\end{thm}
\begin{proof}
Suppose, towards an eventual contradiction, that there is a diagonal
function $D$ and fixed unitaries $U$ and $V$ such that $D(z)=U\Psi(z)V^{*}$.
$D$ must be unitary valued on $B$, so must be unitary valued at
$p_{0}^{-}$, so by multiplying on the left by $D(p_{0}^{-})^{*}$,
we may assume that $D(p_{0}^{-})=I$. Since $\Psi(p_{0}^{-})=I$,
$U=V$.

Let \[
D=\left(\begin{array}{cc}
\phi_{1} & 0\\
0 & \phi_{2}\end{array}\right)\,.\]
 Since $D$ is unitary on $B$, both $\phi_{1}$ and $\phi_{2}$ are
unimodular on $B$. Further, as $\det\Psi$ has $2n+2$ zeroes (up
to multiplicity), and a non-constant scalar inner function has at
least $n+1$ zeroes, we conclude that either $\phi_{1}$ and $\phi_{2}$
have $n+1$ zeroes each, and take each value in the unit disc $\mathbb{D}$
at least $n+1$ times, or one has $2n+2$ zeroes, and the other is
a unimodular constant $\lambda$. The latter cannot occur, since \[
0=\Psi(b)=U^{*}\left(\begin{array}{cc}
\lambda & \cdot\\
\cdot & \cdot\end{array}\right)U\neq0\,,\]
 which would be a contradiction.

Now, from Lemma \vref{lem:PsiSp}, $\Psi(\mathbf{p}_{1})e_{1}=e_{1}$,
so $Ue_{1}$ is an eigenvector of $D(\mathbf{p}_{1})$, corresponding
to the eigenvalue $1$, so at least one of the $\phi_{j}(\mathbf{p}_{1})$s
is equal to $1$. Similarly, $Ue_{2}$ is an eigenvector of $D(\varpi(\mathbf{p}_{1}))$,
so at least one of the $\phi_{j}(\varpi(\mathbf{p}_{1}))$s is equal
to $1$. Now, $D(\mathbf{p}_{1})$ cannot be a multiple of the identity,
as this would mean that one of the $\phi_{j}$s was equal to $1$
at $\mathbf{p}_{1}$ and $\varpi(\mathbf{p}_{1})$, which is impossible%
\footnote{ as this would mean it took the value $1$ at least once on $B_{0},\, B_{2},\,\ldots,\, B_{n}$,
and at least twice on $B_{1}$, so at least $n+2$ times.%
}. Therefore, we can assume without loss of generality that \[
D(\mathbf{p}_{1})=\left(\begin{array}{cc}
1 & 0\\
0 & \lambda\end{array}\right)\,,\qquad D(\varpi(\mathbf{p}_{1}))=\left(\begin{array}{cc}
\lambda^{\prime} & 0\\
0 & 1\end{array}\right)\,,\]
 where $\lambda,\,\lambda^{\prime}$ are unimodular constants. We
can see from this that the eigenvectors corresponding to $1$ in these
matrices are $e_{1}$ and $e_{2}$, so $Ue_{1}=ue_{1}$, $Ue_{2}=u^{\prime}e_{2}$
for unimodular constants $u,\, u^{\prime}$. Since $D$ is diagonal,
we can assume that $u=u^{\prime}=1$, so $U=I$, and $\Psi=D$.

Now, since $S_{\mathbf{M}}\neq S_{0}$, there exists some $i$ such
that $P^{i+}\neq P^{1+}$, so these two projections must have different
ranges. However by Lemma \ref{lem:PsiSp}, \begin{align*}
P^{i+}= & \Psi(\mathbf{p}_{i})\, P^{i+}\\
= & D(\mathbf{p}_{i})\, P^{i+}\\
= & \left(\begin{array}{cc}
\phi_{1}(\mathbf{p}_{i}) & 0\\
0 & \phi_{2}(\mathbf{p}_{i})\end{array}\right)P^{i+}\,.\end{align*}
 This is only possible if $\Psi(\mathbf{p}_{i})=I$, but this is impossible,
as before. This is our contradiction. Therefore, $\Psi$ is not diagonalisable. 
\end{proof}
This concludes the proof of Theorem \ref{thm:ElTheoremGrande}, and
this paper.

\bibliographystyle{amsalpha} \bibliographystyle{amsalpha}
\bibliography{library}

\end{document}